\documentclass{amsart}
\usepackage[margin=1in]{geometry}
\usepackage{amsthm}
\usepackage[foot]{amsaddr}

\usepackage[margin=1in]{geometry}
\usepackage{amsfonts,amsmath,amssymb}
\usepackage{bm}
\usepackage{booktabs}
\usepackage{enumitem}
\usepackage{graphicx}
\usepackage{hyperref}
\usepackage{cleveref} 
\usepackage{stmaryrd}
\usepackage{tikz}
\usepackage{tikz-cd}
\usepackage{xspace}

\usepackage{siunitx}
\usepackage[giveninits,date=year,sortcites]{biblatex}
\def\myqed{\qedhere}

\let\div\undefined
\DeclareMathOperator{\div}{div}
\DeclareMathOperator{\curl}{\mathbf{curl}}
\DeclareMathOperator{\grad}{\mathbf{grad}}
\DeclareMathOperator{\diag}{diag}
\def\Hcurl{\bm{H}(\curl)}
\def\Hdiv{\bm{H}(\div)}
\def\HcurlOmega{\bm{H}(\curl, \Omega)}
\def\HdivOmega{\bm{H}(\div, \Omega)}
\def\Nedelec{N\'ed\'elec\xspace}

\def\NI{\mathit{NI}}

\def\Vp{V_p}
\def\Vh{V_h}
\def\Wp{\bm{W}_p}
\def\Wh{\bm{W}_h}
\def\Xp{\bm{X}_p}
\def\Xh{\bm{X}_h}
\def\Ypm{Y_{p-1}}
\def\Yp{Y_p}
\def\Yh{Y_h}
\def\Zp{Z_p}
\def\Zh{Z_h}

\def\refinterval{[-1,1]}

\def\llb{\llbracket}
\def\rrb{\rrbracket}

\def\iii{\vert\kern-0.25ex\vert\kern-0.25ex\vert}

\newtheorem{prop}{Proposition}
\newtheorem{cor}{Corollary}
\newtheorem{lem}{Lemma}
\newtheorem{thm}{Theorem}
\theoremstyle{definition}
\newtheorem{defn}{Definition}
\theoremstyle{remark}
\newtheorem{rem}{Remark}

\Crefname{prop}{Proposition}{Propositions}
\crefname{thm}{theorem}{theorems}

\title[Low-order preconditioning for high-order de Rham]{Low-order preconditioning for the high-order finite element de~Rham complex}

\author[Pazner]{Will Pazner$^1$}
\author[Kolev]{Tzanio Kolev$^1$}
\address{$^1$Center for Applied Scientific Computing, Lawrence Livermore National Laboratory, Livermore, CA}
\author[Dohrmann]{Clark Dohrmann$^2$}
\address{$^2$Computational Solid Mechanics and Structural Dynamics Department, Sandia National Laboratories, Albuquerque, NM}

\begin{document}

\begin{abstract}
   In this paper we present a unified framework for constructing spectrally equivalent low-order-refined discretizations for the high-order finite element de Rham complex.
This theory covers diffusion problems in $H^1$, $\Hcurl$, and $\Hdiv$, and is based on combining a low-order discretization posed on a refined mesh with a high-order basis for \Nedelec and Raviart--Thomas elements that makes use of the concept of polynomial histopolation (polynomial fitting using prescribed mean values over certain regions).
This spectral equivalence, coupled with algebraic multigrid methods constructed using the low-order discretization, results in highly scalable matrix-free preconditioners for high-order finite element problems in the full de Rham complex.
Additionally, a new lowest-order (piecewise constant) preconditioner is developed for high-order interior penalty discontinuous Galerkin (DG) discretizations, for which spectral equivalence results and convergence proofs for algebraic multigrid methods are provided.
In all cases, the spectral equivalence results are independent of polynomial degree and mesh size; for DG methods, they are also independent of the penalty parameter.
These new solvers are flexible and easy to use; any ``black-box'' preconditioner for low-order problems can be used to create an effective and efficient preconditioner for the corresponding high-order problem.
A number of numerical experiments are presented, using the finite element library MFEM, with which the construction of such preconditioners requires only one or two lines of code.
The theoretical properties of these preconditioners are corroborated, and the flexibility and scalability of the method are demonstrated on a range of challenging three-dimensional problems.
\end{abstract}
\maketitle

\section{Introduction}

High-order finite element methods posed on the discrete de Rham complex (cf.\ \cite{Arnold2006,Arnold2006a}) are of increasing relevance for a wide range of computational applications \cite{Anderson2020,Holec2022,Demkowicz1998,Arnold2006b,White2006}.
Problems such as electromagnetic diffusion (Maxwell's equations), radiation--diffusion transport, and porous media flow require the solution of finite element problems posed in $\Hcurl$ and $\Hdiv$ spaces (using \Nedelec and Raviart--Thomas elements, respectively) \cite{Monk2003,Rieben2007,Brunner2013}.
Discontinuous Galerkin (DG) discretizations of incompressible fluid flow \cite{Krank2017,Shahbazi2007} and radiative transfer \cite{Haut2020} require the solution of interior penalty-type DG discretizations of elliptic problems posed on $L^2$ finite element spaces.
In addition to their attractive accuracy properties, high-order discretizations also lend themselves well to efficient implementations on modern computing architectures \cite{Kolev2021,Kolev2021a,Fischer2020}; in particular, sum-factorization techniques allow for efficient operator evaluation on tensor-product (quadrilateral and hexahedral) meshes at high-orders \cite{Melenk2001,Orszag1980}.
However, the efficient iterative solution of the linear systems that result from these discretizations remains a challenging problem.
The condition number of these linear systems typically scales like $\mathcal{O}(p^3/h^2)$ or $\mathcal{O}(p^4/h^2)$, where $p$ is the polynomial degree, and $h$ is the mesh size, cf.\ \cite{Maitre1996,Melenk2002}, necessitating the use of scalable preconditioners that are robust with respect to the polynomial degree.
Furthermore, while sum factorization techniques give efficient operator evaluation, it is typically not feasible to assemble and store the associated system matrix, ruling out the direct use of matrix-based preconditioning techniques such as incomplete factorizations and algebraic multigrid.
This motivates the development of \textit{matrix-free} preconditioners, that can be constructed without explicit access to the matrix entries of the high-order operator.
Matrix-free solvers for continuous Galerkin discretizations of elliptic problems have been well studied \cite{Kronbichler2018,Kronbichler2019,Bastian2019,Ljungkvist2017,Pazner2020a}.
In this paper, we develop scalable matrix-free preconditioners for high-order finite element discretizations posed on $\Hcurl$, $\Hdiv$, and $L^2$ finite element spaces using a low-order-refined methodology.
Our work extends this well know approach (see below) from $H^1$ to the full de Rham complex, utilizing and complementing the results from \cite{Gerritsma2010} and \cite{Dohrmann2021a}.

Low-order preconditioning is a classical technique for preconditioning high-order and spectral discretizations of the Poisson problem.
This approach was first proposed by Orszag in 1980, who studied the use of standard finite difference methods as preconditioners for spectral methods \cite{Orszag1980}.
The idea was further developed, using low-order finite element methods as a preconditioner for Chebyshev spectral and pseudospectral methods, independently by Deville and Mund \cite{Deville1985,Deville1990} and Canuto and Quarteroni \cite{Canuto1985,Canuto1994}; see \cite{Canuto2010} for a review of these techniques.
Fischer and Lottes developed Schwarz solvers for the pressure solver of the incompressible Navier--Stokes equations using low-order preconditioning in \cite{Fischer1997,Lottes2005}.
Pazner and Kolev applied low-order preconditioning to high-order continuous and discontinuous Galerkin methods with (nonconforming) $hp$-refinement \cite{Pazner2020a,Pazner2021b}.
The aforementioned works use tensor-product elements (mapped quadrilaterals and hexahedra); Chalmers and Warbuton \cite{Chalmers2018} and Olson \cite{Olson2007} considered the extension to simplex elements.

The key property that enables the low-order preconditioning of high-order methods is the spectral equivalence of the low-order and high-order operators \cite{Canuto2006,Canuto2007}.
This property is often known as the finite element method--spectral element method (FEM--SEM) equivalence, and was shown in \cite{Canuto1994,Parter1995}.
The extension of these low-order equivalence results from $H^1$ to $\Hcurl$ and $\Hdiv$ is more involved, in part because of the nontrivial nullspaces of the curl and divergence operators (consisting of irrotational and solenoidal vector fields, respectively).
In contrast, in the case of $H^1$ finite element spaces, the kernel of the gradient operator consists only of constant functions.
In this paper, we give a unified presentation of high-order--low-order spectral equivalence for all spaces in the de Rham complex.
These spectral equivalences are based on properties of one-dimensional polynomial interpolation and histopolation operators (polynomial histopolation is the process of finding a polynomial with given integrals over a certain number of disjoint intervals).
The interpolation and histopolation operators give rise to specially designed high-order bases for \Nedelec and Raviart--Thomas elements.
These bases were introduced in \cite{Gerritsma2010}, and their spectral equivalence properties and applications to preconditioning were studied in \cite{Dohrmann2021a}.

The structure of the paper is as follows.
High-order and low-order interpolation and histopolation in one spatial dimension are discussed in \Cref{sec:interp-histop}.
The extension to differential operators in multiple spatial dimensions is then considered in \Cref{sec:multi-d}.
The application of these properties to the spectral equivalence of high-order and low-order-refined finite element discretizations is presented in \Cref{sec:fe}.
The performance of matrix-free preconditioners for high-order problems based on this spectral equivalence is studied numerically in \Cref{sec:numerical-results};
in this section, we consider scalable preconditioners using algebraic multigrid methods applied to large-scale three-dimensional problems.
We end with conclusions in \Cref{sec:conclusions}.

\section{One-dimensional interpolation and histopolation} \label{sec:interp-histop}

The high-order--low-order-refined equivalence depends on the appropriate choice of basis for the high-order finite element spaces.
In this work, we construct bases for the $H^1$, $\Hcurl$ (\Nedelec), $\Hdiv$ (Raviart--Thomas), and $L^2$ finite element spaces on meshes with tensor-product elements that are designed to satisfy this equivalence.
These bases are defined using the concepts of polynomial \textit{interpolation} and \textit{histopolation} on the reference 1D interval $\refinterval$.
Interpolation defines a nodal (Lagrange) basis for the space $\mathcal{Q}_p$ of degree-$p$ polynomials in $\refinterval$ using a set of $p+1$ distinct nodal points and values.
Histopolation, on the other hand, is a procedure that defines a degree-$(p-1)$ polynomial in terms of its integrals over $p$ distinct subintervals.

\begin{defn}[Interpolation]
   Let $\{ x_i \}_{i=1}^{p+1} \subseteq \refinterval$ be a set of $p+1$ distinct nodes.
   Then, the associated \textbf{interpolation operator} $\mathcal{I}_p : \mathbb{R}^{p+1} \to \mathcal{Q}_p$ maps prescribed point values $\mathsf{u} = \{ \mathsf{u}_i \}_{i=1}^{p+1}$ to the unique polynomial interpolant $u = \mathcal{I}_p(\mathsf{u}) \in \mathcal{Q}_p$ such that $u(x_i) = \mathsf{u}_i$ for all $i$.
\end{defn}

\begin{defn}[Histopolation]
   Let $\{ I_i = (x_i^L, x_i^R ) \}_{i=1}^{p}$ be a set of $p$ disjoint subintervals of $\refinterval$, and let $h_i$ denote the size of the $i$th interval, $h_i = x_i^R - x_i^L$.
   Then, the associated \textbf{histopolation operator} $\mathcal{H}_{p-1} : \mathbb{R}^{p} \to \mathcal{Q}_{p-1}$ maps prescribed average values $\widehat{\mathsf{u}} = \{ \widehat{\mathsf{u}}_i \}_{i=1}^{p}$ to the unique polynomial histopolant $u = \mathcal{H}_{p-1}(\widehat{\mathsf{u}}) \in \mathcal{Q}_{p-1}$ such that $\frac{1}{h_i} \int_{x_i^L}^{x_i^R} u(x) \, dx = \widehat{\mathsf{u}}_i$ for all $i$.
\end{defn}

\begin{prop}
   The interpolation and histopolation operators $\mathcal{I}_p$ and $\mathcal{H}_{p-1}$ are well-defined linear bijections.
\end{prop}
\begin{proof}
   The operator $\mathcal{I}_p$ can be defined through the classical Lagrange interpolation procedure.
   For the case of $\mathcal{H}_{p-1}$, it suffices to show that for a given $u \in \mathcal{Q}_{p-1}$, if $\int_{x_i^L}^{x_i^R} u \, dx = 0$ for all $i$ then $u \equiv 0$.
   This is a well-known and simple result, which we repeat here (cf.\ \cite{Chihara1978}, see also \cite{Kolev2022} for a generalization).
   On each of the disjoint intervals $I_i$, since $\int_{x_i^L}^{x_i^R} u \, dx = 0$, either $u \equiv 0$ or $u$ changes sign on $I_i$.
   In the latter case, $u \in \mathcal{Q}_{p-1}$ has $p$ distinct zeros, and so $u \equiv 0$.
\end{proof}

\subsection{Gauss--Lobatto points}
In what follows, we choose the interpolation points and histopolation subintervals to be compatible, in the sense that the $p+1$ interpolation points $\{ x_i \}_{i=1}^{p+1}$ also define the $p$ subintervals by $\{ (x_i, x_{i+1} ) \}_{i=1}^p$.
In particular, we choose $x_i$ to be the $p+1$ Gauss--Lobatto points on the reference interval, which are given by the zeros of $(1 - x^2) P_p'(x)$, where $P_p(x)$ is the degree-$p$ Legendre polynomial.
Since these points include the interval endpoints, the associated subintervals form a partition of $\refinterval$.
Additionally, these points, as well as others such as the Gauss--Legendre and Chebyshev points, are asymptotically distributed according to the Chebyshev density $(p + 1) / (\pi \sqrt{1 - x^2})$ \cite{Szego1939,Berrut2004}.

\begin{rem}[Notation]
   We will write $a \lesssim b$ to mean that $a \leq C b$, where $C$ is independent of the polynomial degree $p$ (and other discretization parameters, such as the mesh size $h$, when relevant). $a \gtrsim b$ means $b \lesssim a$, and $a \approx b$ means both $a \lesssim b$ and $b \lesssim a$.
   Similarly, for two symmetric and positive-definite matrices $A$ and $B$, $A \sim B$ means that $A$ and $B$ are spectrally equivalent, and the constants of equivalence are independent of $p$ (and other discretization parameters).
\end{rem}

\subsection{Low-order equivalences} \label{sec:1d-lor-equiv}

Analogous to the high-order interpolation and histopolation operators, we also define low-order interpolation and histopolation operators, $\mathcal{I}_h$ and $\mathcal{H}_h$.
Given the interpolation points $\{ x_i \}$, let $V_h$ denote the space of piecewise linear functions with nodes at each point $x_i$.
Similarly, given the $p$ disjoint subintervals $\{ (x_i, x_{i+1}) \}$, we define $\widehat{V}_h$ to be the space of piecewise constant functions spanned by the indicator functions of these subintervals.
Let $\mathcal{I}_h : \mathbb{R}^{p+1} \to V_h$ denote the piecewise linear interpolation operator, and let $\mathcal{H}_h : \mathbb{R}^{p} \to \widehat{V}_h$ denote the operator that maps average values to the piecewise constant function in $\widehat{V}_h$ taking those values over each subinterval.
Then, we have the following equivalences in the $L^2(\refinterval)$-norm (cf.\ \cite{Canuto1994,Kolev2022}).

\begin{prop} \label{prop:1d-l2-equiv}
   It holds that
   \begin{equation} \label{eq:1d-l2-equiv}
      \| \mathcal{I}_h \mathsf{u} \|_0^2 \approx \| \mathcal{I}_p \mathsf{u} \|_0^2
      \qquad
      \text{for all $\mathsf{u} \in \mathbb{R}^{p+1}$},
   \end{equation}
   \begin{equation} \label{eq:1d-l2-histop-equiv}
      \| \mathcal{H}_h \widehat{\mathsf{u}} \|_0^2 \approx \| \mathcal{H}_{p-1} \widehat{\mathsf{u}} \|_0^2
      \qquad
      \text{for all $\widehat{\mathsf{u}} \in \mathbb{R}^p$}.
   \end{equation}
\end{prop}
\begin{proof}
   The proof of \eqref{eq:1d-l2-equiv} is given in \cite[Prop.\ 2.1]{Canuto1994}.
   The proof of \eqref{eq:1d-l2-histop-equiv} is given in \cite[Prop.\ 5]{Kolev2022}.
\end{proof}

\begin{rem}
   It additionally holds that $\| \mathcal{H}_h \mathsf{u} \|_0^2 \approx \| \mathcal{I}_{p-1} \mathsf{u} \|_0^2$ for all $\mathsf{u} \in \mathbb{R}^p$.
   This equivalence directly compares the interpolation and histopolation operators applied to the same vector of values, and is a consequence of the asymptotic equivalence between the Gauss--Lobatto quadrature weights and subinterval lengths (cf.\ \cite[Proposition 3]{Kolev2022}).
\end{rem}

\begin{rem} \label{rem:basis}
   The operators $\mathcal{I}_p$ and $\mathcal{H}_{p-1}$ induce bases for the polynomial spaces $\mathcal{Q}_p$ and $\mathcal{Q}_{p-1}$, where each basis vector is the image of a standard Cartesian basis vector under the aforementioned operators.
   Similarly, the operator $\mathcal{I}_h$ induces the standard basis of ``hat functions'' on the piecewise linear space $V_h$, and $\mathcal{H}_h$ induces the basis of piecewise constant indicator functions on the space $\widehat{V}_h$.

   The high-order interpolatory (nodal) basis functions $\ell_j \in \mathcal{Q}_p$ are given by the Lagrange interpolating polynomials of the nodal points $x_i$.
   The high-order histopolatory basis functions $\vartheta_j \in \mathcal{Q}_{p-1}$ satisfy $\int_{x_i}^{x_{i+1}} \vartheta_j(x) \, dx = h_i \delta_{ij}$.
   It is straightforward to see that these functions are given by the negative partial sum of the derivatives of the nodal basis functions,
   \[
      \vartheta_i(x) = - h_i \sum_{k=1}^i \ell_k'(x).
   \]
   Graphical plots of these basis functions are shown in \cite{Dohrmann2021a}.
\end{rem}

The high-order and low-order interpolants and histopolants are closely related through their derivatives.
For any $f \in H^1$ that interpolates nodal values $(x_i, y_i)$, it is straightforward to see that $f'$ has mean value $m_i = (y_{i+1} - y_i)/(x_{i+1} - x_i)$ over the interval $[x_i, x_{i+1}]$.
As a consequence, the derivative $\mathcal{I}_p( y_i )'$ of the nodal interpolant of $f$ is equal to the histopolant $\mathcal{H}_{p-1}(m_i)$, and similarly $\mathcal{I}_h(y_i)' = \mathcal{H}_h(m_i)$.
This gives the following natural relationship between the derivatives of the high-order and low-order interpolants.
Let $I_h : \mathcal{Q}_p \to V_h$ and $H_h : \mathcal{Q}_{p-1} \to \widehat{V}_h$ be the high-order to low-order interpolation and histopolation operators, defined by $I_h = \mathcal{I}_h \mathcal{I}_p^{-1}$ and $H_h = \mathcal{H}_h \mathcal{H}_{p-1}^{-1}$, respectively.
Similarly, let $I_p = I_h^{-1}$ and $H_p = H_h^{-1}$ denote the low-order to high-order operators.

\begin{prop} \label{prop:deriv-histop}
   Let $\mathsf{u} \in \mathbb{R}^{p+1}$ be given, and let $u_p = \mathcal{I}_p(\mathsf{u}) \in \mathcal{Q}_p$, and $u_h = \mathcal{I}_h(\mathsf{u}) \in V_h$.
   Then,
   \[
      \mathcal{H}_{p-1}^{-1} u_p' = \mathcal{H}_h^{-1} u_h',
   \]
   and so
   \[
      u_h' = I_h(u_p)' = H_h (u_p')
      \qquad\text{and}\qquad
      u_p' = I_p(u_h)' = H_p (u_h').
   \]
   Equivalently, the following diagram commutes, where $\partial$ represents the derivative operator.
   \[
      \begin{tikzcd}
         \mathcal{Q}_p \arrow[d, "I_h"] \arrow[r, "\partial"] & \mathcal{Q}_{p-1} \arrow[d, "H_h"] \\
         V_h \arrow[r, "\partial"] & \widehat{V}_h
      \end{tikzcd}
   \]
\end{prop}
\begin{proof}
   $\int_{x_i}^{x_{i+1}} u_h'(x) \, dx = u_h(x_{i+1}) - u_h(x_i) = \mathsf{u}_{i+1} - \mathsf{u}_i = u_p(x_{i+1}) - u_p(x_i) = \int_{x_i}^{x_{i+1}} u_p'(x) \, dx.$
\end{proof}

\begin{rem} \label{rem:interp-histop}
   A restatement of the above proposition is that interpolating and then differentiating is equivalent to differentiating and then histopolating, using both the high-order and low-order operators.
   Consequently, from the perspective of high-order--low-order-refined equivalence, it is natural to represent functions by interpolating their point values at nodal points, and it is natural to represent their derivatives by histopolating average values of subintervals defined by the same nodal points.
\end{rem}

\begin{rem} \label{rem:interp-histop-conformity}
   Interpolation with nodes at the interval endpoints naturally allows for continuity across interfaces (i.e.\ when constructing conforming spaces), whereas histopolation does not.
   Therefore it is natural to use interpolation for $H^1$ spaces, and for the continuous tangential components in $\Hcurl$ and normal components in $\Hdiv$.
   Histopolation may be used for the components for which continuity is not enforced and for DG spaces.
\end{rem}

The above results can also be combined to give the equivalence of the $H^1$-seminorm of the high-order and low-order interpolants.

\begin{cor} \label{cor:1d-h1-equiv}
   $\| \left(\mathcal{I}_h \mathsf{u}\right)' \|_0^2 \approx \| \left(\mathcal{I}_p \mathsf{u}\right)' \|_0^2$ for all $\mathsf{u} \in \mathbb{R}^{p+1}$.
\end{cor}
\begin{proof}
   Let $\widehat{\mathsf{u}} = \mathcal{H}_{p-1}^{-1} u_p' = \mathcal{H}_h^{-1} u_h'$ (by \Cref{prop:deriv-histop}).
   Then, $\left(\mathcal{I}_h \mathsf{u}\right)' = \mathcal{H}_h \widehat{\mathsf{u}}$ and $\left(\mathcal{I}_{p} \mathsf{u}\right)' = \mathcal{H}_{p-1} \widehat{\mathsf{u}}$, and so the result follows from \Cref{prop:1d-l2-equiv}.
\end{proof}

\section{Equivalences in multiple dimensions} \label{sec:multi-d}

In \Cref{sec:1d-lor-equiv}, the norm equivalence for the one-dimensional high-order and low-order interpolation and differentiation operators was established.
In this section, we extend the construction of interpolation and histopolation operators to multiple dimensions using a tensor-product construction.
Additionally, we prove the analogous norm equivalence properties for the high-order and low-order interpolation, gradient, curl, and divergence operators.

The interpolation operator on the $d$-dimensional reference element $\refinterval^d$ is defined by
\[
   \mathcal{I}_p^d : \mathbb{R}^{(p+1)^d} \to \mathcal{Q}_p(\refinterval^d),
   \qquad
   \mathcal{I}_p^d = \mathcal{I}_p \otimes \cdots \otimes \mathcal{I}_p.
\]
This operator maps point values defined on the Cartesian product of nodal points to their unique multivariate interpolating polynomial.
The low-order piecewise multilinear interpolation operator $\mathcal{I}_h^d$ can be defined analogously, and similarly for the histopolation operators $\mathcal{H}_{p-1}$ and $\mathcal{H}_h$.

\begin{prop} \label{prop:l2-equiv}
   Let $\mathsf{u} \in \mathbb{R}^{(p+1)^d}$ and $\widehat{\mathsf{u}} \in \mathbb{R}^{p^d}$.
   Then, $\| \mathcal{I}_h^d \mathsf{u} \|_0^2 \approx \| \mathcal{I}_p^d \mathsf{u} \|_0^2$ and $\| \mathcal{H}_h^d \widehat{\mathsf{u}} \|_0^2 \approx \| \mathcal{H}_{p-1}^d \widehat{\mathsf{u}} \|_0^2$.
\end{prop}
\begin{proof}
   The proof is an immediate consequence of \Cref{prop:1d-l2-equiv} and properties of the tensor product.
\end{proof}

For concreteness and ease of notation, we will focus on the case of $d=3$ for the remainder of this section.
Many of the results presented here are generalizable to the case of arbitrary $d$ in a straightforward manner.

\subsection{Gradient operators}
These interpolation operators naturally give rise to high-order and low-order gradient operators,
\[
   \mathcal{G}_p^d = \nabla \mathcal{I}_p^d,
   \qquad\text{and}\qquad
   \mathcal{G}_h^d = \nabla \mathcal{I}_h^d.
\]
The norm equivalences of the previous section can be extended to show the equivalence of the high-order and low-order gradient operators.

\begin{prop} \label{prop:grad-equiv}
   Let $\mathsf{u} \in \mathbb{R}^{(p+1)^d}$.
   Then, $\| \mathcal{G}_h^d \mathsf{u} \|_0^2 \approx \| \mathcal{G}_p^d \mathsf{u} \|_0^2$.
\end{prop}
\begin{proof}
   For ease of notation, we write the proof here for $d=3$, although the extension to arbitrary $d$ is straightforward.
   Note that $\mathcal{G}_p^3$ and $\mathcal{G}_h^3$ are given by
   \[
      \mathcal{G}_p^3 = \nabla \mathcal{I}_p^3 = \left(\begin{array}{c}
         \left( \partial \mathcal{I}_p \right) \otimes \mathcal{I}_p \otimes \mathcal{I}_p \\
         \mathcal{I}_p \otimes \left( \partial \mathcal{I}_p \right) \otimes \mathcal{I}_p \\
         \mathcal{I}_p \otimes \mathcal{I}_p \otimes \left( \partial \mathcal{I}_p \right)
      \end{array} \right),
      \qquad
      \mathcal{G}_h^3 = \nabla \mathcal{I}_h^3 = \left(\begin{array}{c}
         \left( \partial \mathcal{I}_h \right) \otimes \mathcal{I}_h \otimes \mathcal{I}_h \\
         \mathcal{I}_h \otimes \left( \partial \mathcal{I}_h \right) \otimes \mathcal{I}_h \\
         \mathcal{I}_h \otimes \mathcal{I}_h \otimes \left( \partial \mathcal{I}_h \right)
      \end{array} \right).
   \]
   By \Cref{prop:deriv-histop}, we have $\partial \mathcal{I}_h = \mathcal{H}_h \mathcal{H}_{p-1}^{-1} \partial \mathcal{I}_p$, and therefore can write
   \[
      \mathcal{G}_h^3 = \left(\begin{array}{ccc}
         \mathcal{H}_h \mathcal{H}_{p-1}^{-1} \otimes \mathcal{I}_h\mathcal{I}_p^{-1} \otimes \mathcal{I}_h\mathcal{I}_p^{-1} & 0 & 0\\
         0 & \mathcal{I}_h\mathcal{I}_p^{-1} \otimes \mathcal{H}_h \mathcal{H}_{p-1}^{-1} \otimes \mathcal{I}_h\mathcal{I}_p^{-1} & 0\\
         0 & 0 & \mathcal{I}_h\mathcal{I}_p^{-1} \otimes \mathcal{I}_h\mathcal{I}_p^{-1} \otimes \mathcal{H}_h \mathcal{H}_{p-1}^{-1}
      \end{array}\right)
      \mathcal{G}_p^3.
   \]
   From \Cref{prop:1d-l2-equiv}, we have that $\| \mathcal{I}_h \mathcal{I}_p^{-1} u \|^2 \approx \| u \|_0^2$ for all $u \in \mathcal{Q}_p$ and $\| \mathcal{H}_h \mathcal{H}_{p-1}^{-1} u \|^2 \approx \| v \|_0^2$ for all $v \in \mathcal{Q}_{p-1}$.
   The result then follows from properties of the tensor product.
\end{proof}

\subsection{Curl operators}
In light of \Cref{rem:interp-histop} and the mapping $H^1 \xrightarrow{\grad} \Hcurl$, it is natural to represent discrete functions in $\Hcurl$ using the histopolation and interpolation operators
\begin{gather*}
   \mathcal{I}_p^{\curl} = \left(\begin{array}{ccc}
      \mathcal{H}_{p-1} \otimes \mathcal{I}_p \otimes \mathcal{I}_p & 0 & 0 \\
      0 & \mathcal{I}_p \otimes \mathcal{H}_{p-1} \otimes \mathcal{I}_p & 0 \\
      0 & 0 & \mathcal{I}_p \otimes \mathcal{I}_p \otimes \mathcal{H}_{p-1}
   \end{array}\right), \\
   \mathcal{I}_h^{\curl} = \left(\begin{array}{ccc}
      \mathcal{H}_{h} \otimes \mathcal{I}_h \otimes \mathcal{I}_h & 0 & 0 \\
      0 & \mathcal{I}_h \otimes \mathcal{H}_{h} \otimes \mathcal{I}_h & 0 \\
      0 & 0 & \mathcal{I}_h \otimes \mathcal{I}_h \otimes \mathcal{H}_{h}
   \end{array}\right).
\end{gather*}
The high-order and low-order curl operators are then defined naturally as
\[
   \mathcal{C}_p = \nabla \times \mathcal{I}_p^{\curl},
   \qquad\text{and}\qquad
   \mathcal{C}_h = \nabla \times \mathcal{I}_h^{\curl}.
\]
\begin{prop} \label{prop:curl-equiv}
   Let $\mathsf{u} \in \mathbb{R}^{p (p+1)^2}$.
   Then, $\| \mathcal{I}_p^{\curl} \mathsf{u} \|_0^2 \approx \| \mathcal{I}_h^{\curl} \mathsf{u} \|_0^2$ and $ \| \mathcal{C}_h \mathsf{u} \|_0^2 \approx \| \mathcal{C}_p \mathsf{u} \|_0^2$.
\end{prop}
\begin{proof}
   The first equivalence is a simple consequence of \Cref{prop:1d-l2-equiv}.
   To show the second equivalence, we write the operators $\mathcal{C}_p$ and $\mathcal{C}_h$ explicitly as
   \begin{gather*}
      \mathcal{C}_p = \left(\begin{array}{ccc}
         0 &
         -\mathcal{I}_p \otimes \mathcal{H}_{p-1} \otimes (\partial \mathcal{I}_p) &
         \mathcal{I}_p \otimes (\partial \mathcal{I}_p) \otimes \mathcal{H}_{p-1} \\
         \mathcal{H}_{p-1} \otimes \mathcal{I}_p \otimes (\partial \mathcal{I}_p) &
         0 &
         -(\partial \mathcal{I}_p) \otimes \mathcal{I}_p \otimes \mathcal{H}_{p-1} \\
         -\mathcal{H}_{p-1} \otimes (\partial \mathcal{I}_p) \otimes \mathcal{I}_p &
         (\partial \mathcal{I}_p) \otimes \mathcal{H}_{p-1} \otimes \mathcal{I}_p &
         0
      \end{array}\right), \\
      \mathcal{C}_h = \left(\begin{array}{ccc}
         0 &
         -\mathcal{I}_h \otimes \mathcal{H}_{h} \otimes (\partial \mathcal{I}_h) &
         \mathcal{I}_h \otimes (\partial \mathcal{I}_h) \otimes \mathcal{H}_{h} \\
         \mathcal{H}_{h} \otimes \mathcal{I}_h \otimes (\partial \mathcal{I}_h) &
         0 &
         -(\partial \mathcal{I}_h) \otimes \mathcal{I}_h \otimes \mathcal{H}_{h} \\
         -\mathcal{H}_{h} \otimes (\partial \mathcal{I}_h) \otimes \mathcal{I}_h &
         (\partial \mathcal{I}_h) \otimes \mathcal{H}_{h} \otimes \mathcal{I}_h &
         0
      \end{array}\right),
   \end{gather*}
   from which the relation
   \[
      \mathcal{C}_h
      =
      \left(\begin{array}{ccc}
         \mathcal{I}_h \mathcal{I}_p^{-1} \otimes \mathcal{H}_h \mathcal{H}_{p-1}^{-1} \otimes \mathcal{H}_h \mathcal{H}_{p-1}^{-1} & 0 & 0 \\
         0 & \mathcal{H}_h \mathcal{H}_{p-1}^{-1} \otimes \mathcal{I}_h \mathcal{I}_p^{-1} \otimes \mathcal{H}_h \mathcal{H}_{p-1}^{-1} & 0 \\
         0 & 0 & \mathcal{H}_h \mathcal{H}_{p-1}^{-1} \otimes \mathcal{H}_h \mathcal{H}_{p-1}^{-1} \otimes \mathcal{I}_h \mathcal{I}_p^{-1}
      \end{array}\right)
      \mathcal{C}_p
   \]
   is clear.
   The conclusion then follows as in the proof of \Cref{prop:grad-equiv}.
\end{proof}

\subsection{Divergence operators}
Given the mapping $\Hcurl \xrightarrow{\curl} \Hdiv$, the interpolation operators in $\Hdiv$ are naturally given by
\begin{gather*}
   \mathcal{I}_p^{\div} = \left(\begin{array}{ccc}
      \mathcal{I}_p \otimes \mathcal{H}_{p-1} \otimes \mathcal{H}_{p-1} & 0 & 0 \\
      0 & \mathcal{H}_{p-1} \otimes \mathcal{I}_p \otimes \mathcal{H}_{p-1} & 0 \\
      0 & 0 & \mathcal{H}_{p-1} \otimes \mathcal{H}_{p-1} \otimes \mathcal{I}_p
   \end{array}\right), \\
   \mathcal{I}_h^{\div} = \left(\begin{array}{ccc}
      \mathcal{I}_{h} \otimes \mathcal{H}_h \otimes \mathcal{H}_h & 0 & 0 \\
      0 & \mathcal{H}_h \otimes \mathcal{I}_{h} \otimes \mathcal{H}_h & 0 \\
      0 & 0 & \mathcal{H}_h \otimes \mathcal{H}_h \otimes \mathcal{I}_{h}
   \end{array}\right),
\end{gather*}
and the divergence operators by
\[
   \mathcal{D}_p = \nabla \cdot \mathcal{I}_p^{\div},
   \qquad\text{and}\qquad
   \mathcal{D}_h = \nabla \cdot \mathcal{I}_h^{\div}.
\]
\begin{prop} \label{prop:div-equiv}
   Let $\mathsf{u} \in \mathbb{R}^{p^2 (p+1)}$.
   Then, $ \| \mathcal{I}_p^{\div} \mathsf{u} \|_0^2 \approx \| \mathcal{I}_h^{\div} \mathsf{u} \|_0^2 $ and $\| \mathcal{D}_h \mathsf{u} \|_0^2 \approx \| \mathcal{D}_p \mathsf{u} \|_0^2$.
\end{prop}
\begin{proof}
   The first equivalence is a simple consequence of \Cref{prop:1d-l2-equiv}.
   To show the second equivalence, we write the operators $\mathcal{D}_p$ and $\mathcal{D}_h$ as
   \begin{gather*}
      \mathcal{D}_p = \left(\begin{array}{ccc}
         (\partial \mathcal{I}_p) \otimes \mathcal{H}_{p-1} \otimes \mathcal{H}_{p-1} &
         \mathcal{H}_{p-1} \otimes (\partial \mathcal{I}_p) \otimes \mathcal{H}_{p-1} &
         \mathcal{H}_{p-1} \otimes \mathcal{H}_{p-1} \otimes (\partial \mathcal{I}_p)
      \end{array}\right), \\
      \mathcal{D}_h = \left(\begin{array}{ccc}
         (\partial \mathcal{I}_h) \otimes \mathcal{H}_h \otimes \mathcal{H}_h &
         \mathcal{H}_h \otimes (\partial \mathcal{I}_h) \otimes \mathcal{H}_h &
         \mathcal{H}_h \otimes \mathcal{H}_h \otimes (\partial \mathcal{I}_h)
      \end{array}\right),
   \end{gather*}
   from which it can be seen that $\mathcal{D}_h = \mathcal{H}^d_h (\mathcal{H}^d_{p-1})^{-1} \mathcal{D}_p,$ and the conclusion follows.
\end{proof}

\begin{rem}
   As an immediate consequence of the above propositions, the high-order and low-order gradient, curl, and divergence operators have identical nullspaces.
   This is particularly important for the construction of preconditioners for discretizations in $\Hcurl$ and $\Hdiv$, for which treating the nontrivial nullspaces of the curl and divergence operators is a key challenge.
\end{rem}

\begin{rem}
   The spectral equivalence results of \Crefrange{prop:l2-equiv}{prop:div-equiv} were first proven using explicit calculations in \cite{Dohrmann2021a}.
   In this section, we have provided alternative, systematic proofs of these results based only on the one-dimensional interpolation and histopolation operators.
\end{rem}

\section{Finite element spaces}
\label{sec:fe}

The results of the previous sections can be used to define spectral equivalences of operators defined on $H^1$, $\Hcurl$, $\Hdiv$, and $L^2$ finite element spaces, which make up the discrete de Rham complex in 3D (cf.~\cite{Arnold2006,Arnold2006a}):
\[
   \begin{tikzcd}
      H^1(\Omega) \arrow[r, "\grad"] &
      \HcurlOmega \arrow[r, "\curl"] &
      \HdivOmega \arrow[r, "\div"] &
      L^2(\Omega).
   \end{tikzcd}
\]
We begin by considering a spatial domain $\Omega \subseteq \mathbb{R}^3$ (analogous results also apply in a straightforward manner to domains in one and two spatial dimensions).
The domain is discretized using a hexahedral mesh $\mathcal{T}_p$.
Each element of the mesh $\kappa \in \mathcal{T}_p$ is the image of the reference element $\widehat{\kappa} = \refinterval^3$ under a smooth mapping, $\kappa = T_\kappa (\widehat{\kappa})$.
Given the mesh $\mathcal{T}_p$, we define the following high-order and low-order-refined finite element spaces.

\subsection{High-order spaces}

Recall the standard $H^1$, $\Hcurl$ (\Nedelec), $\Hdiv$ (Raviart--Thomas), and $L^2$ finite element spaces, defined as follows.
Note that in addition to the natural $L^2$ finite element space in the de Rham complex, we also define a discontinuous Galerkin spaces that can be considered as a ``broken $H^1$'' space.

\begin{itemize}
   \item The $H^1$ finite element space
   \[
      \Vp = \{\, v \in H^1(\Omega) : v|_\kappa \circ T_\kappa \in \mathcal{Q}_p(\widehat{\kappa}) \,\}.
   \]
   \item The $\Hcurl$ finite element space with \Nedelec elements
   \[
      \Wp = \{\, \bm{w} \in \Hcurl : \bm{w}|_\kappa \circ T_\kappa \in \Wp(\kappa)  \,\}.
   \]
   The local space $\Wp(\kappa)$ is the image of the reference space $\Wp(\widehat{\kappa}) = \mathcal{Q}_{p-1,p,p} \times \mathcal{Q}_{p,p-1,p} \times \mathcal{Q}_{p,p,p-1}$ under the $\Hcurl$ Piola transformation $\Wp(\kappa) = J_\kappa^{-T} \Wp(\widehat{\kappa})$, where $J_\kappa$ is the Jacobian matrix of the element transformation $T_\kappa$.
   \item The $\Hdiv$ finite element space with Raviart--Thomas elements
   \[
      \Xp = \{\, \bm{x} \in \Hdiv : \bm{w}|_\kappa \circ T_\kappa \in \Xp(\kappa) \  \,\}.
   \]
   The local space $\Xp(\kappa)$ is the image of the reference space $\Xp(\widehat{\kappa}) = \mathcal{Q}_{p,p-1,p-1} \times \mathcal{Q}_{p-1,p,p-1} \times \mathcal{Q}_{p-1,p-1,p}$ under the $\Hdiv$ Piola transformation $\Xp(\kappa) = \det(J_\kappa)^{-1} J_\kappa \Xp(\widehat{\kappa})$, where $J_\kappa$ is the Jacobian matrix of the element transformation $T_\kappa$.
   \item The $L^2$ finite element space
   \[
      \Ypm = \{\, y \in L^2(\Omega) : y|_\kappa \circ T_\kappa \in \Ypm(\kappa)  \,\}.
   \]
   The local space $\Ypm(\kappa)$ is the image of the reference space $\Ypm(\widehat{\kappa}) = \mathcal{Q}_{p-1}$ under the integral-preserving mapping, $\Ypm(\kappa) = \det(J_{\kappa})^{-1} \Ypm(\widehat{\kappa})$, where $J_\kappa$ is the Jacobian matrix of the element transformation $T_\kappa$.
   \item The discontinuous Galerkin space
   \[
      \Zp = \{\, z \in L^2(\Omega) : z|_\kappa \circ T_\kappa \in \mathcal{Q}_{p}(\widehat{\kappa}) \,\}.
   \]
   In contrast to the $L^2$ space $\Yp$, the local DG space does not incorporate the integral-preserving mapping.
   For our purposes, it is more natural to consider the DG space $\Zp$ as a ``broken $H^1$'' finite element space rather than an $L^2$ space.
\end{itemize}

The spaces $\Vp$, $\Wp$, $\Xp$, $\Ypm$ provide a discrete analogue of the $L^2$ de Rham complex, in the sense that
\[
   \begin{tikzcd}
      \Vp \arrow[r, "\grad"] &
      \Wp \arrow[r, "\curl"] &
      \Xp \arrow[r, "\div"] &
      \Ypm
   \end{tikzcd}
\]
is a complete sequence, where the range of each operator is exactly the kernel of the next one, e.g.\ $\curl \bm{w} = \bm{0}$ for $\bm{w}\in\Wp$ if and only if $\bm{w} = \grad v$ for some $v \in \Vp$.

\subsection{Low-order-refined spaces}

For each of the spaces $\Vp, \Wp, \Xp, \Ypm, \Zp$, the corresponding low-order-refined spaces $\Vh, \Wh, \Xh, \Yh, \Zp$ are given by the lowest-order finite element spaces, defined on a Gauss--Lobatto refined mesh.
This refined mesh, denoted $\mathcal{T}_h$, is obtained by refining each element $\kappa \in \mathcal{T}_p$ as follows.
Let $\{ x_i \}_{i=1}^{p+1}$ denote the $p+1$ Gauss--Lobatto points in $\refinterval$, and let $\{ \bm{x}_i \}_{i=1}^{(p+1)^3}$ denote their 3-fold Cartesian product in $\refinterval^3$.
The points $\bm{x}_i$ define a submesh of the element $\kappa$, consisting of $p^3$ subelements.
This submesh is structured (Cartesian), but nonuniform because of the clustering of the Gauss--Lobatto points at the endpoints of the interval.
The $H^1$ space $\Vh$ is the space of piecewise trilinear functions defined on $\mathcal{T}_h$, with degrees of freedom given by vertex values.
The $\Hcurl$ space $\Wh$ is the space of lowest-order edge elements, and the $\Hdiv$ space $\Xh$ is the space of lowest-order face elements defined on $\mathcal{T}_h$.
The $L^2$ space $\Yh$ is the space of piecewise constants defined on $\mathcal{T}_h$.
Similarly, the low-order DG space $\Zh$ (corresponding to the degree-$p$ DG space $\Zp$) is the space of piecewise constants defined on the mesh $\mathcal{T}_h'$ obtained by refining $\mathcal{T}_p$ using $p+2$ Gauss--Lobatto points.
An illustration of a low-order-refined hexahedral mesh is shown in \Cref{fig:fichera}.

\begin{figure}
   \centering
   \includegraphics[width=1.5in]{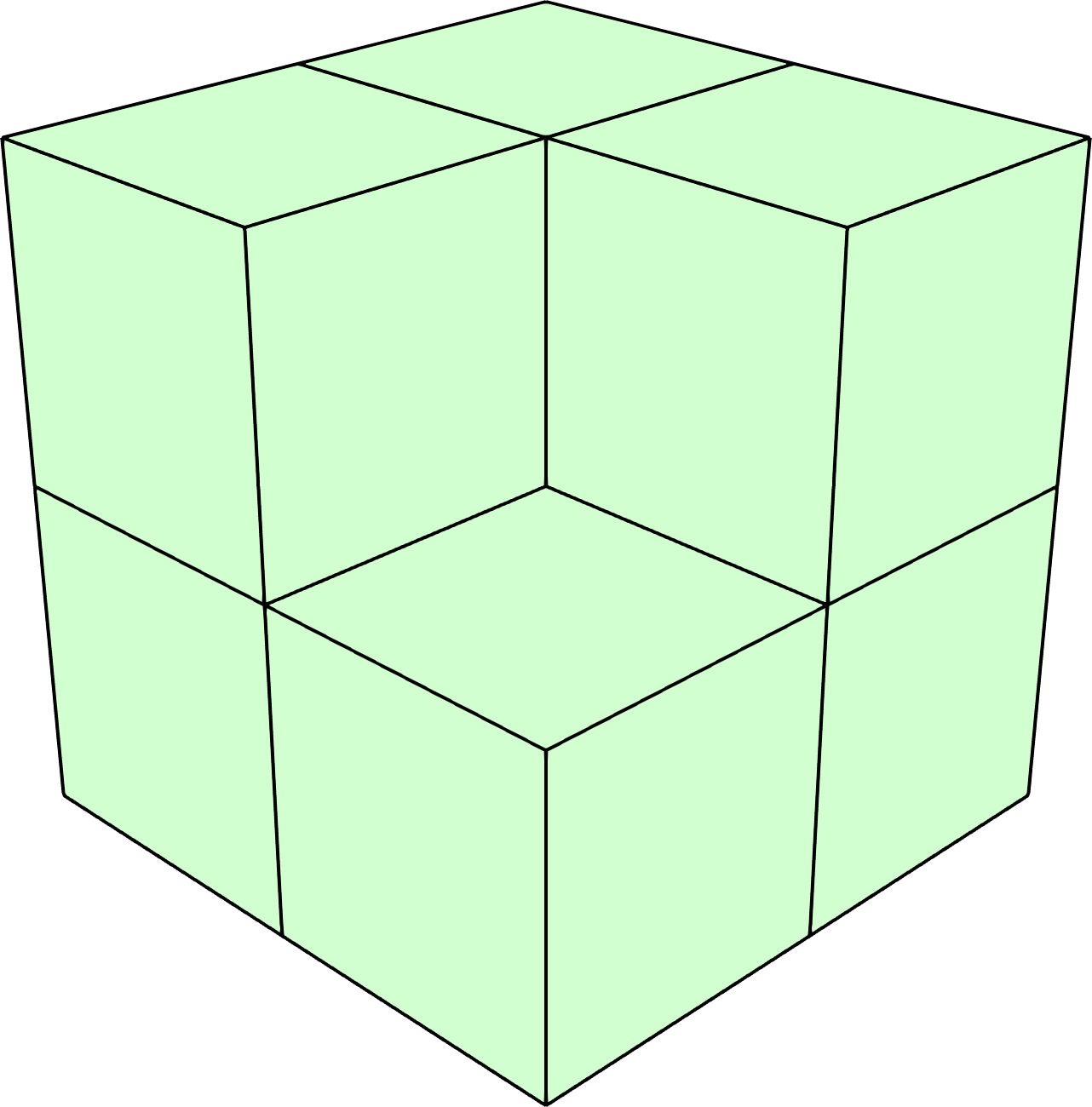}
   \hspace{0.25in}
   \includegraphics[width=1.5in]{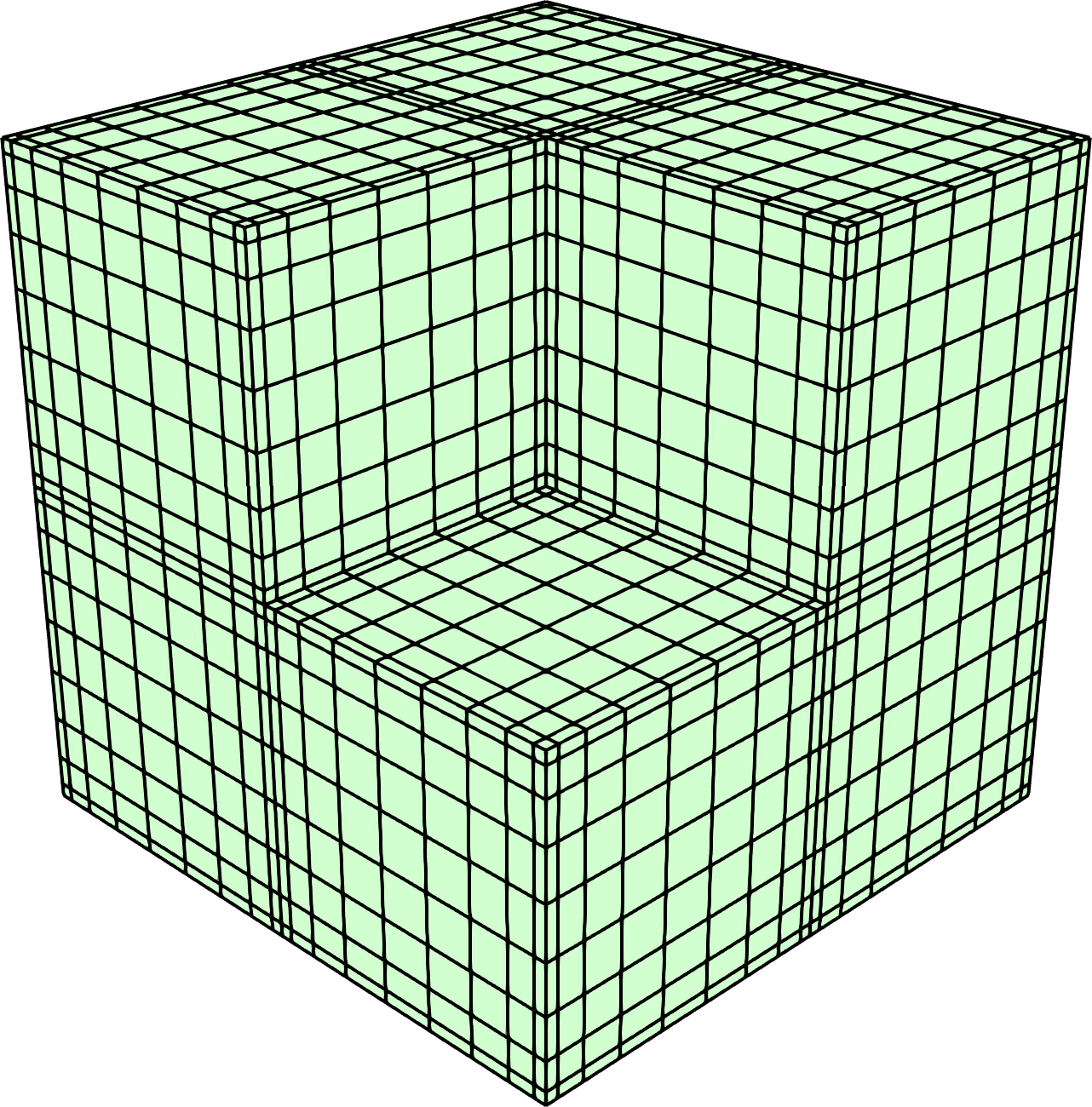}
   \caption{
      Hexahedral mesh of the Fichera corner (left), and its low-order-refined counterpart corresponding to $p=9$ (right).
   }
   \label{fig:fichera}
\end{figure}

\subsection{Choice of basis for the high-order spaces} \label{sec:basis}

To obtain the high-order--low-order-refined equivalences between the finite element spaces, we make use of bases built using the interpolation and histopolation operators from \Cref{sec:interp-histop,sec:multi-d}.
These bases are closely related to the so-called \textit{mimetic bases} introduced by Gerritsma and colleagues in \cite{Gerritsma2010,Kreeft2011,Zhang2018}, and studied in the context of high-order--low-order equivalence in \cite{Dohrmann2021a}.

Recall the multi-dimensional interpolation and histopolation operators defined in \Cref{sec:multi-d},
\begin{align*}
   \mathcal{I}_p^3 &= \mathcal{I}_p \otimes \mathcal{I}_p \otimes \mathcal{I}_p, \\
   \mathcal{I}_p^{\curl} &= \left(\begin{array}{ccc}
      \mathcal{H}_{p-1} \otimes \mathcal{I}_p \otimes \mathcal{I}_p & 0 & 0 \\
      0 & \mathcal{I}_p \otimes \mathcal{H}_{p-1} \otimes \mathcal{I}_p & 0 \\
      0 & 0 & \mathcal{I}_p \otimes \mathcal{I}_p \otimes \mathcal{H}_{p-1}
   \end{array}\right), \\
   \mathcal{I}_p^{\div} &= \left(\begin{array}{ccc}
      \mathcal{I}_p \otimes \mathcal{H}_{p-1} \otimes \mathcal{H}_{p-1} & 0 & 0 \\
      0 & \mathcal{H}_{p-1} \otimes \mathcal{I}_p \otimes \mathcal{H}_{p-1} & 0 \\
      0 & 0 & \mathcal{H}_{p-1} \otimes \mathcal{H}_{p-1} \otimes \mathcal{I}_p
   \end{array}\right), \\
   \mathcal{H}_{p-1}^3 &= \mathcal{H}_{p-1} \otimes \mathcal{H}_{p-1} \otimes \mathcal{H}_{p-1}.
\end{align*}
The images of the Cartesian basis vectors $\mathsf{e}^i$ (where $\mathsf{e}_j^i = \delta_{ij}$) under each of the above mappings naturally define basis functions for the corresponding space of polynomials.

\begin{itemize}
   \item The basis functions defined by $\mathcal{I}_p^3 : \mathbb{R}^{(p+1)^3} \to \mathcal{Q}_p$ are the standard nodal (Lagrange) basis functions corresponding to the Gauss--Lobatto points.
   By identifying coincident nodal points at element interfaces, the resulting functions are continuous across element interfaces, giving a basis for $\Vp$.
   In this case, the degrees of freedom are point values at the nodes.
   \item The basis functions defined by $\mathcal{I}_p^{\curl} : \mathbb{R}^{p(p+1)^2} \to \mathcal{Q}_{p-1,p,p} \times \mathcal{Q}_{p,p-1,p} \times \mathcal{Q}_{p,p,p-1}$ are used to define $\Hcurl$ basis functions.
   In this case, the degrees of freedom of the $i$th vector component are integrals over the segment connecting two neighboring nodes in the $i$th coordinate direction.
   This naturally results in tangential continuity, ensuring that the resulting piecewise polynomial functions form a basis for $\Wp$.
   \item The basis functions defined by $\mathcal{I}_p^{\div} : \mathbb{R}^{p^2(p+1)} \to \mathcal{Q}_{p,p-1,p-1} \times \mathcal{Q}_{p-1,p,p-1} \times \mathcal{Q}_{p-1,p-1,p}$ are used to define $\Hdiv$ basis functions.
   The resulting degrees of freedom for the $i$th vector component are interals over the two-dimensional surface defined by neighboring nodes in the two orthogonal coordinate directions.
   This naturally results in normal continuity, ensuring that the resulting piecewise polynomial functions form a basis for $\Xp$.
   \item The basis functions defined by $\mathcal{H}_{p-1}^3 : \mathbb{R}^{p^3} \to \mathcal{Q}_{p-1}$ are used for the $L^2$ space $\Ypm$.
   The resulting degrees of freedom represent integrals over subvolumes defined by the nodal points, enforcing no continuity between elements.
   \item For the discontinuous Galerkin space $\Zp$, the basis functions defined by $\mathcal{I}_{p-1}^3$ are used (as in the case of $H^1$ elements).
   The case of DG spaces is discussed in further detail in \Cref{sec:dg}.
\end{itemize}

\begin{rem}
   The degrees of freedom for the high-order finite element spaces described above coincide exactly with the standard lowest-order degrees of freedom for the low-order-refined spaces.
   Consequently, in the lowest-order cases ($p=1$ for $H^1$, $\Hcurl$, and $\Hdiv$ finite elements, $p=0$ for $L^2$ finite elements), the basis functions described above reduce to the standard basis functions used for the lowest-order finite element spaces, and we recover the standard low-order vertex, edge, face, and element basis functions.
\end{rem}

\subsection{High-order--low-order-refined spectral equivalence}

In this section, we consider the spectral equivalence of the mass and stiffness matrices defined on the finite element spaces $\Vp, \Wp, \Xp,$ and $\Ypm$.
For simplicity, we restrict the analysis in this section to the case where the mesh element transformations $T_\kappa$ have constant Jacobians (in other words, $T_\kappa$ is an affine transformation, and $\kappa$ is a parallelepiped).
The case of more general meshes (and variable coefficients) is studied numerically in \Cref{sec:numerical-results}.

We define the following transfer operators between the high-order and low-order spaces.
\[
   \begin{aligned}
      P_{V} &: \Vp \to \Vh
         \quad&\quad
         P_{V}(v)|_{\kappa} &= \mathcal{I}_h^3 \left( \mathcal{I}_p^3 \right)^{-1} (u|_{\kappa}) \\
      P_{W} &: \Wp \to \Wh
         \quad&\quad
         P_{W}(\bm{w})|_{\kappa} &= \mathcal{I}_h^{\curl} \left( \mathcal{I}_p^{\curl} \right)^{-1} (\bm{w}|_{\kappa}) \\
      P_{X} &: \Xp \to \Xh
         \quad&\quad
         P_{X}(\bm{x})|_{\kappa} &= \mathcal{I}_h^{\div} \left( \mathcal{I}_p^{\div} \right)^{-1} (\bm{x}|_{\kappa}) \\
      P_{Y} &: \Ypm \to \Yh
         \quad&\quad
         P_{Y}(y)|_{\kappa} &= \mathcal{H}_h^3 \left( \mathcal{H}_{p-1}^3 \right)^{-1} (y|_{\kappa}) \\
      P_{Z} &: \Zp \to \Zh
         \quad&\quad
         P_{Z}(z)|_{\kappa} &= \mathcal{I}_h^3 \left( \mathcal{I}_{p}^3 \right)^{-1} (z|_{\kappa}) \\
   \end{aligned}
\]
The definition of these transfer operators (together with the standard commuting projection operators $\Pi_\circ$ \cite{Monk2003} ensure that the following diagram commutes:
\[
   \begin{tikzcd}
      H^1(\Omega) \arrow[d, "\Pi_V"] \arrow[r, "\grad"] &
      \HcurlOmega \arrow[d, "\Pi_W"] \arrow[r, "\curl"] &
      \HdivOmega  \arrow[d, "\Pi_X"] \arrow[r, "\div"] &
      L^2(\Omega) \arrow[d, "\Pi_Y"] \\
      \Vp  \arrow[d, "P_V"] \arrow[r, "\grad"] &
      \Wp  \arrow[d, "P_W"] \arrow[r, "\curl"] &
      \Xp  \arrow[d, "P_X"] \arrow[r, "\div"] &
      \Ypm \arrow[d, "P_Y"] \\
      \Vh \arrow[r, "\grad"] &
      \Wh \arrow[r, "\curl"] &
      \Xh \arrow[r, "\div"] &
      \Yh
   \end{tikzcd}
\]
\begin{rem}
   Note that because of the choice of basis laid out in \Cref{sec:basis}, the matrix representation of each of the above $P$ operators is the identity matrix.
   In other words, the same vector of degrees of freedom represents both element of the high-order finite element space, and its image under the transfer operator in the low-order-refined space.
\end{rem}

These transfer operators result in the following norm and seminorm equivalences, which immediately give the spectral equivalences of the high-order and low-order-refined mass and stiffness matrices.
\begin{thm} \label{thm:norm-equiv}
   It holds that
   \begin{gather}
      \begin{alignat}{4}
         \| v \|_0^2 &\approx \| P_V (v) \|_0^2, \quad&\quad
         | v |_1^2 &\approx | P_V (v) |_1^2 \qquad&&\text{for all $v \in \Vp$},\\
         \| \bm{w} \|_0^2 &\approx \| P_W (\bm{w}) \|_0^2, \quad&\quad
         | \bm{w} |_{\curl}^2 &\approx | P_W (\bm{w}) |_{\curl}^2 \qquad&&\text{for all $\bm{w} \in \Wp$},\\
         \| \bm{x} \|_0^2 &\approx \| P_X (\bm{x}) \|_0^2, \quad&\quad
         | \bm{x} |_{\div}^2 &\approx | P_X (\bm{x}) |_{\div}^2 \qquad&&\text{for all $\bm{x} \in \Xp$},\\
         \| y \|_0^2 &\approx \| P_Y (y) \|_0^2 &&&&\text{for all $y \in \Ypm$},\\
         \| z \|_0^2 &\approx \| P_Z (z) \|_0^2, \quad&\quad
         \iii z \iii_p^2 &\approx \iii P_Z (z) \iii_h^2 \qquad&&\text{for all $z \in \Zp$},
      \end{alignat}
   \end{gather}
   where $| \bm{w} |_{\curl} = \| \nabla \times \bm{w} \|_0$ and $| \bm{x} |_{\div} = \| \nabla \cdot \bm{x} \|_0$ denote the curl and divergence seminorms, respectively, and $\iii \cdot \iii_p$ and $\iii \cdot \iii_h$ denote the high-order and low-order mesh-dependent DG norms (defined in \Cref{sec:dg}).
\end{thm}
\begin{proof}
   The proof proceeds easily by summing over each element $\kappa \in \mathcal{T}_p$, using the assumption of constant Jacobians, and the norm equivalences established in \Cref{prop:l2-equiv,prop:grad-equiv,prop:curl-equiv,prop:div-equiv}.
   The equivalence in the DG norms $\iii \cdot \iii_p$ and $\iii \cdot \iii_h$ is deferred to \Cref{sec:dg}.
\end{proof}

\begin{thm} \label{thm:spectral-equivalence}
   Let $M_\star$ and $K_\star$ denote the mass and stiffness matrices, respectively, where $\star$ represents one of the above-defined finite element spaces with basis as in \Cref{sec:basis}.
   Then we have the following spectral equivalences, independent of mesh size $h$ and polynomial degree $p$.
   \[
      \begin{aligned}
         M_{\Vh} &\sim M_{\Vp}, \quad& K_{\Vh} &\sim K_{\Vp}, \\
         M_{\Wh} &\sim M_{\Wp}, \quad& K_{\Wh} &\sim K_{\Wp}, \\
         M_{\Xh} &\sim M_{\Xp}, \quad& K_{\Xh} &\sim K_{\Xp}, \\
         M_{\Yh} &\sim M_{\Ypm}, \\
         M_{\Zh} &\sim M_{\Zp}, \quad& K_{\Zh} &\sim K_{\Zp}.
      \end{aligned}
   \]
\end{thm}
\begin{proof}
   These spectral equivalences follow immediately from the norm equivalences of \Cref{thm:norm-equiv}.
\end{proof}

\subsection{Discontinuous Galerkin discretizations} \label{sec:dg}

In the context of the DG space $\Zp$, since no continuity is enforced between elements, it is possible to use either the interpolatory basis induced by $\mathcal{I}_p^d$ or the histopolation basis induced by $\mathcal{H}_p^d$.
Both of theses basis give rise of $L^2$ norm equivalence for the high-order and low-order spaces.
However, the histopolation basis does not give a straightforward low-order equivalence for interior penalty discretizations of the diffusion operator.
For this reason, we prefer to use the interpolatory (nodal) basis induced by $\mathcal{I}_p^d$ for DG spaces.
This is the same basis that is used for $H^1$ spaces, which is natural if we interpret the DG space as a ``broken $H^1$'' space rather than an $L^2$ finite element space.

The DG low-order-refined mesh $\mathcal{T}_h'$ is obtained by subdividing each element $\kappa \in \mathcal{T}_p$ into $(p+1)^d$ subelements, defined by the Cartesian product of the $p+2$ Gauss--Lobatto points $x_i'$.
The nodal basis of $\Zp$ is defined using the $p+1$ Gauss--Lobatto points $x_i$.
Note that the interlacing property of the Gauss--Lobatto quadrature implies that in every interval $[x_i', x_{i+1}']$ there lies exactly one point $x_i$ \cite{Szego1939,Beckermann2013}.
The transfer operator $P_Y$ then maps piecewise polynomials $u_p$ to piecewise constant functions whose constant value over each subcell is given by the value of $u_p$ at the unique nodal point lying in that subcell.

Consider the symmetric interior penalty (IP) discretization of the Poisson problem \cite{Arnold1982,Arnold2002}
\begin{equation} \label{eq:ip}
   \mathcal{A}_{\Zp} (u, v)
   = ( \nabla_h u, \nabla_h v)
      - \langle \{ \nabla_h u \}, \llb v \rrb \rangle
      - \langle \llb u \rrb, \{ \nabla_h v \} \rangle
      + \langle \sigma_p \llb u \rrb, \llb v \rrb \rangle
   = (f, v),
\end{equation}
where $\nabla_h$ denotes the broken gradient operator and $\langle \cdot\,,\cdot \rangle$ denotes integration over element interfaces (i.e.\ over the mesh skeleton $\Gamma_p$ of $\mathcal{T}_p$).
The notation $\{\,\cdot\,\}$ and $\llb\,\cdot\,\rrb$ is used to denote the average and jump of a function at element interfaces, respectively.
In the above, $\sigma_p = \eta p^2 / h$ is the \textit{penalty parameter}, which must be chosen sufficiently large to obtain a stable method.
The norm induced by IP discretization is equivalent to mesh dependent norm $\iii\cdot\iii_p$, defined by
\[
\iii u \iii_p^2 = \| \nabla_h \|_0^2 + \| \sigma_p^{1/2} \llb u \rrb \|_{0,\Gamma_p}^2.
\]
\begin{prop}[Cf.\ \cite{Antonietti2010,Houston2002}] \label{prop:dg-norm}
   For all $u, v \in \Zp$, it holds that
   \begin{align*}
      \mathcal{A}_{\Zp}(u, v) &\lesssim \iii u \iii_p \, \iii v \iii_p, \\
      \mathcal{A}_{\Zp}(u, u) &\gtrsim \iii u \iii_p^2.
   \end{align*}
\end{prop}

This norm equivalence allows for the construction of a spectrally equivalent low-order ($p=0$) discretization defined on the refined mesh $\mathcal{T}_h$.
We first note that $\mathcal{A}_{\Zp}$ restricted to the piecewise constant space $\Zh$ reduces to only the penalty term $\langle \sigma_p \llb u \rrb, \llb v \rrb \rangle$, since $\nabla_h u = 0$ for all $u \in \Zh$.
Therefore, we define the low-order interior penalty form
\begin{equation} \label{eq:ip-lor}
   \mathcal{A}_{\Zh} = \langle \sigma_h \llb u \rrb, \llb v \rrb \rangle,
\end{equation}
where the integrals are performed over $\Gamma_h$, the mesh skeleton of $\mathcal{T}_h$.
The choice of penalty parameter $\sigma_h$, defined as a piecewise constant field on each face $e \in \Gamma_h$, is of critical importance.
The form $\mathcal{A}_{\Zh}$ induces the low-order DG norm $\iii u \iii_h = \| \sigma_h^{1/2} \llb u \rrb \|_{0,\Gamma_h}$.

We first write the low-order mesh skeleton $\Gamma_h$ as the disjoint union $\Gamma_h = \Gamma_\circ \cup \Gamma_\partial$.
The set $\Gamma_\partial$ consists of faces that are subsets of coarse mesh faces, i.e.\ $e_\partial \in \Gamma_\partial$ satisfies $e_\partial \subseteq e_p$ for some $e_p \in \Gamma_p$.
On the other hand, the set $\Gamma_\circ$ denotes those faces that lie in the interior of the coarse high-order macro-elements.
A face $e_\circ \in \Gamma_\circ$ satisfies $e_\circ \nsubseteq e_p$ for all $e_p \in \Gamma_p$.
The piecewise constant penalty parameter $\sigma_h(e)$ will be defined separately for $e \in \Gamma_\circ$ and $e \in \Gamma_\partial$.

Consider the reference macro-element $\widehat{\kappa} = \refinterval^3$, which has been decomposed into $(p+1)^3$ subelements.
On general meshes, a scaling factor $\alpha$ is defined by multiplying by the ratio of the reference and physical element sizes, $\alpha = \widehat{h} / h$.
The element size $h$ at the face is computed as the average of the sizes $h = \frac{1}{2}(h^+ + h^-)$ of the adjacent elements $\kappa^{\pm}$, where $h^{\pm}$ is computed as the perpendicular length of the element, $h^{\pm} = \mu(\kappa^{\pm})/\mu(e)$, where $\mu$ denotes measure.

First we consider an interior face $e_{\circ} \in \Gamma_\circ$, and, without loss of generality, we assume that $e_{\circ}$ is normal to the $z$ coordinate direction.
In this case, $e_{\circ} = [x_i', x_{i+1}'] \times [x_j', x_{j+1}'] \times \{ x_k' \}$ for some $(i,j,k)$.
Since $e_{\circ}$ is an interior face, we have $1 < k < p+2$.
Let $x_i$ and $x_j$ denote the unique nodal points lying in $[x_i', x_{i+1}']$ and $[x_j', x_{j+1}']$, respectively, and let $w_i$ and $w_j$ denote the corresponding Gauss--Lobatto weights.
Additionally, let $x_k$ and $x_{k+1}$ denote the unique nodal points lying in $[x_{k-1}', x_k']$ and $[x_k', x_{k+1}']$, respectively.
Then, define
\begin{equation} \label{eq:sigma-int}
   \sigma_h(e_{\circ}) = \frac{\alpha w_i w_j}{(x_{i+1}' - x_i')(x_{j+1}' - x_j')(x_{k+1}-x_k)}.
\end{equation}

Now, consider the case of $e_{\partial} \in \Gamma_\partial$.
Without loss of generality, write $e_\partial = [x_i', x_{i+1}'] \times [x_j', x_{j+1}'] \times \{ 0 \}$.
As before, let $w_i$ and $w_j$ denote the Gauss--Lobatto weights corresponding to the unique nodes lying in the intervals $[x_i', x_{i+1}']$ and $[x_j', x_{j+1}']$.
Then, define
\begin{equation} \label{eq:sigma-bdr}
   \sigma_h(e_\partial) = \alpha \eta p^2 w_i w_j.
\end{equation}

\def\L{{L}}
\def\G{{G}}
\def\E{{E}}
\def\V{{V}}
\def\e{{e}}
\def\T{{T}}
\def\D{{D}}

\begin{rem} \label{rem:graph-laplacian}
   The discretization defined by \eqref{eq:ip-lor} is equivalent to the \textit{weighted graph Laplacian} defined on the connectivity graph of the mesh $\mathcal{T}_h$.
   Let $\G = (\V, \E)$ be the graph defined by $\mathcal{T}_h$, such that each element $\kappa_i \in \mathcal{T}_h$ corresponds to a vertex $i \in \V$, and the edge $(i,j) \in \E$ exists whenever elements $\kappa_i, \kappa_j \in \mathcal{T}_h$ share a common face $e \in \Gamma_h$, in which case we also write $i \sim j$.

   For each graph edge $(i,j) \in \E$, define the weight $w_{ij}$ by $w_{ij} = \sigma_h(e)$ (given by \eqref{eq:sigma-int} and \eqref{eq:sigma-bdr}), where $e \in \Gamma_h$ is the interface between elements $\kappa_i$ and $\kappa_j$.
   For each graph vertex $i$, the weight $w_i$ is defined by $w_i = \sum_{i \sim j} w_{ij}$.
   Then, the weighted graph Laplacian $\L$ of $\G$ is the matrix defined by
   \begin{equation} \label{eq:graph-laplacian}
      \L_{ij} = \begin{cases}
         -w_{ij} & \text{if $i \sim j$,} \\
         w_i & \text{if $i = j$,} \\
         0 & \text{otherwise.}
      \end{cases}
   \end{equation}
   It is straightforward to see that the matrix $\L$ is identical to the stiffness matrix $K_{\Zh}$ corresponding to the bilinear form $\mathcal{A}_{\Yp}$ in the case of Neumann boundary conditions.
\end{rem}

\begin{thm} \label{thm:dg-equivalence}
   The low-order DG discretization defined by \eqref{eq:ip-lor} with $\sigma_h$ given by \eqref{eq:sigma-int} and \eqref{eq:sigma-bdr} is spectrally equivalent to the high-order DG discretization \eqref{eq:ip}.
\end{thm}
\begin{proof}
   Let $u_p \in \Zp$ be given, and let $u_h = P_Z u_p$.
   Then, by \Cref{prop:dg-norm},
   \[
      \mathcal{A}_{\Zp}(u_p, u_p)
         \approx \iii u_p \iii^2
         = \| \nabla_h u_p \|_0^2
         + \| \sigma_p^{1/2} \llb u_p \rrb \|_{0,\Gamma_p}^2.
   \]
   We first consider the term $\| \llb u \rrb \|_{0,\Gamma_p}^2$.
   For a given face $e_p \in \Gamma_p$, consider the set $\mathcal{E}(e_p) = \{ e \in \Gamma_\partial : e \subseteq e_p \}$.
   Each such subelement face $e_i \in \mathcal{E}(e_p)$ corresponds to a nodal point $x_i \in e_p$ and Gauss--Lobatto weight $w_i$.
   Given definition \eqref{eq:sigma-bdr}, we have $\sigma_h(e_i) = \alpha \eta p^2 w_i$.
   Using the property of Gauss--Lobatto quadrature that $\| f \|_{0,e}^2 \approx \sum w_i f(x_i)^2$ (cf.\ \cite{Canuto1982}), we have
   \begin{align*}
      \sum_{e_i \in \mathcal{E}(e_p)} \| \sigma_h^{1/2} \llb u_h \rrb \|_{e_i}^2
      &= \sum_{e_i \in \mathcal{E}(e_p)} \alpha \eta p^2 w_i \mu(e_i) \llb u_h(x_i) \rrb^2 \\
      &\approx \eta p^2 h^{-1} \sum_{e_i} w_i \mu(e_i) \llb u_p(x_i) \rrb^2 \\
      &\approx \| \sigma_p^{1/2} \llb u_p \rrb \|_{0,e_p},
   \end{align*}
   using that $\alpha \approx h^{-1}$.
   Therefore, $\| \sigma_p^{1/2} \llb u_p \rrb \|_{0,\Gamma_p}^2 \approx \| \sigma_h^{1/2} \llb u_h \rrb \|_{0,\Gamma_\partial}$.

   Now we consider the term $\| \nabla_h u_p \|_{0,\kappa}^2$ on a given element $\kappa \in \mathcal{T}_p$.
   By \Cref{prop:grad-equiv}, we have that $\| \nabla_h u_p \|_{0,\kappa}^2 \approx \| \nabla_h \widetilde{u}_h \|_{0,\kappa}^2$, where $\widetilde{u}_h$ is the piecewise linear interpolant of $u_p$.
   Now, consider an interior face $e_\circ \in \Gamma_\circ$.
   Without loss of generality, assume that $e_\circ$ is the image under the element transformation mapping $T_\kappa$ of $\widehat{e} = [x_i', x_{i+1}'] \times [x_j', x_{j+1}'] \times \{ x_k' \}$  for some $(i,j,k)$; the cases of faces normal to the $x$ and $y$ coordinate directions in the reference element follow analogously.
   On the face $e_\circ$ we have $\frac{\partial \widetilde{u}_h}{\partial z} \approx h^{-1} \frac{\llb u_h \rrb}{x_{k+1} - x_k}$ and so
   \begin{align*}
      \int_{\kappa} \left( \frac{\partial \widetilde{u}_h}{\partial z} \right)^2 \, dx
      \approx \sum_{ijk} h^3 w_i w_j w_k \left( \frac{\partial \widetilde{u}_h}{\partial z} \right)^2
      \approx \sum_{ijk} h w_i w_j w_k \left( \frac{\llb u_h \rrb}{x_{k+1} - x_k} \right)^2.
   \end{align*}
   We compute
   \begin{align*}
      \int_{e_\circ} \sigma_h \llb u_h \rrb^2 \, ds
      &= \frac{\alpha w_i w_j \mu(e_\circ)}{(x_{i+1}' - x_i')(x_{j+1}' - x_j')(x_{k+1}-x_k)} \llb u_h \rrb^2 \\
      &\approx \frac{h w_i w_j}{x_{k+1}-x_k} \llb u_h \rrb^2 \\
      &= h w_i w_j (x_{k+1}-x_k) \left(\frac{\llb u_h \rrb}{x_{k+1} - x_k}\right)^2.
   \end{align*}
   Recalling that $w_k \approx x_{k+1} - x_k$, and summing over all interior faces (including the $x$- and $y$-normal faces), we obtain
   \[
   \sum_{e_\circ \in \Gamma_\circ} \sigma_h \llb u_h \rrb^2 \, ds
   \approx \sum_{\kappa \in \mathcal{T}_p} \| \nabla_h \widetilde{u}_h \|_{0,\kappa}^2
   \approx \| \nabla_h u_p \|_0^2,
   \]
   and the result follows.
\end{proof}

\section{Algebraic multigrid preconditioning}

Let $A_h$ denote a convex combination of the low-order-refined stiffness and mass matrices $K_h$ and $M_h$ (where the subscript $h$ is shorthand for one of $\Vh, \Wh, \Xh, \Yh$, or $\Zh$).
Let $A_p$ denote the associated high-order operator.
The spectral equivalence results of \Cref{thm:spectral-equivalence} (i.e.\ $A_h \sim A_p$) imply that any good preconditioner for the low-order and sparse system $A_h$ will also be a good preconditioner for the corresponding high-order system $A_p$.
In principle, there are a number of multigrid, domain decomposition, and incomplete factorization preconditioners that will result in well-conditioned systems.
In this work, we focus on algebraic multigrid (AMG) methods: these methods give essentially black-box highly scalable preconditioners for $A_h$ requiring minimal discretization information.
AMG convergence for lowest-order $H^1$ and DG finite element discretizations for elliptic problems has been extensively studied in the literature \cite{Brandt1986, McCormick1985, Ruge1987, Falgout2004} and has further been extended to definite $\Hcurl$ \cite{Kolev2009, Brunner2011} and $\Hdiv$ \cite{Kolev2012, Dobrev2019} problems.
Additionally, several high-performance massively parallel and GPU-accelerated implementations such as the BoomerAMG, AMS and ADS preconditioners in the {\em hypre} library \cite{Falgout2002} are available.

Although not studied here, domain decomposition algorithms can also be used as preconditioners for the low-order matrix $A_h$.
Both iterative substructuring and overlapping Schwarz algorithms were analyzed in \cite{Casarin1997} using the FEM--SEM equivalence of $H^1$ discretizations.
That analysis was motivated by earlier numerical experiments in \cite{Pahl1993}.
Overlapping Schwarz preconditioners using the FEM--SEM equivalence were studied recently in \cite{Dohrmann2021a} for both $\Hcurl$ and $\Hdiv$ discretizations.
The numerical results in that study were promising, but the analysis of domain decomposition preconditioners for $A_h$ for problems in these two function spaces remains an open problem.

\subsection{Mass matrix preconditioning} \label{sec:mass-precond}

It is well known that the high-order mass matrix (using either nodal Gauss--Lobatto or Gauss--Legendre basis) is spectrally equivalent to its diagonal, independent of the polynomial degree $p$ \cite{Canuto1982,Canuto2010} (on parallelepiped elements, the Gauss--Legendre matrix is equal to its diagonal).
In fact, it can be shown that on the reference interval $\refinterval$, the fully integrated mass matrix with Gauss--Lobatto basis is given by a rank-one update to the diagonal matrix of Gauss--Lobatto weights \cite{Teukolsky2015}.
In this case, the matrix $D^{-1} M$, where $D = \diag(M)$, has only two distinct eigenvalues, and its condition number \textit{decreases} with increasing $p$.

It is straightforward to show that the mass matrix using the interpolation--histopolation basis defined in this paper is also spectrally equivalent to its diagonal.
A comparison of diagonal preconditioners for the high-order mass matrix is included in \Cref{sec:mass-results}.

\begin{prop}
   Let $M_p$ denote the high-order mass matrix defined on one of the spaces $\Vp$, $\Wp$, $\Xp$, $\Ypm$, or $\Zp$.
   Let $D_p = \diag(M_p)$.
   Then, $M_p \sim D_p$, independent of $p$.
\end{prop}
\begin{proof}
   Let $M_h$ denote the mass matrix defined on the corresponding low-order space.
   By \Cref{thm:spectral-equivalence}, $M_p \sim M_h$.
   But $M_h$ is the standard mass matrix using the lowest-order basis, which is spectrally equivalent to its diagonal, $D_h = \diag(M_h)$.
   Since $M_p \sim M_h$, we also have $D_p \sim D_h$, and so $M_p \sim M_h \sim D_h \sim D_p$.
\end{proof}

\subsection{Discontinuous Galerkin discretizations}

We consider classical AMG methods applied to the low-order-refined DG discretization described in \Cref{sec:dg}.
Note that the graph Laplacian $\L$ defined by \eqref{eq:graph-laplacian} is an M-matrix, for which the convergence of classical algebraic multigrid methods is well-studied \cite{Ruge1987}.
As in the preceding sections, it is then expected that AMG applied to $\L$ will result in convergence that is independent of $h$ and $p$ (modulo the low-order mesh anisotropy).
In this section, we show that classical AMG applied to $\L$ also converges independently of the penalty parameter $\eta$.
We relate this result to the family of preconditioners that use an associated $H^1$-conforming discretization to precondition discontinuous Galerkin methods \cite{Dobrev2006,Antonietti2016,Pazner2020a}.

Recalling the language of \Cref{rem:graph-laplacian}, we consider the graph $\G = (\V, \E)$, and its associated weighted graph Laplacian $\L$.
The edges $\e \in \E$ can be categorized as either \textit{interior edges} $\e \in \E^\circ$, in which case the associated weight is given by \eqref{eq:sigma-int}, or as \textit{boundary edges} $\e \in \E^\partial$, in which case the associated weight is given by \eqref{eq:sigma-bdr}.
The weights associated with interior edges are independent of $\eta$, whereas the weights associated with boundary edges scale linearly with $\eta$.
In the following, we assume that $\eta \gg 1$, i.e.\ the boundary weights dominate the interior weights, and so we will say that two vertices are \textit{strongly connected} if there exists a boundary edge connecting them.

We partition the graph $\G$ into a set of disjoint \textit{strongly connected components}
\[
   \G = \G^{(1)} \cup \G^{(2)} \cup \cdots \cup \G^{(n)} \cup \G^{\circ},
\]
where any two vertices of $\G^{(i)} = (\V^{(i)}, \E^{(i)})$ must be connected by a path consisting of boundary edges.
The strongly connected components of $\G$ correspond to groups of degrees of freedom lying on distinct mesh entities of the coarse mesh $\mathcal{T}_p$.
For example, all the degrees of freedom that are coincident with a mesh vertex belong to the same strongly connected component.
Likewise, the two coincident degrees of freedom lying on the interior of a mesh face belong to the same connected component.
Any degree of freedom lying in the interior of a mesh element $\kappa \in \mathcal{T}_p$ has no strong connections, and such vertices are included in the interior component $\G^{\circ}$.
The number of vertices in a given strongly connected component of $\G$ is bounded by the valence of the coarse mesh $\mathcal{T}_p$, which we assume to be $\mathcal{O}(1)$.

Classical algebraic multigrid methods partition the vertices $\V$ of the graph into coarse (C) points and fine (F) points.
Let $P = \begin{pmatrix} W \\ I \end{pmatrix}$ denote the C-to-F interpolation operator.
Let $Q = P (P^\T P)^{-1} P^\T$ denote the orthogonal projection onto the range of $P$.
Because of the assumption that $\eta \gg 1$ the operator $Q$ is decoupled across the strongly connected components of $\G$.
Let $Q^{(i)}$ and $Q^\circ$ denote the projections corresponding to the subgraphs $\G^{(i)}$ and $\G^\circ$, respectively.
$Q$ is then given as the product of these operators (since the projections are decoupled, this product is commutative).
The operators $Q^{(i)}$ possess two important properties:
\begin{enumerate}[label=(P\arabic*)]
   \item Since the AMG interpolation $P$ preserves constants, $Q^{(i)}$ has row-sum equal to one. \label{item:row-sum}
   \item Since $Q^{(i)}$ is an orthogonal projection, $| Q^{(i)}_{jk} | \leq 1$ for all $j, k$. \label{item:l2-norm}
\end{enumerate}
We proceed to show that the coarse grid defined by the C-points, together with Jacobi relaxation, results in a stable decomposition independent of $\eta$, and hence uniform AMG convergence.
In what follows, let $\D$ denote the diagonal of $\L$.

\begin{lem} \label{lem:interior-decomposition}
   Let $u$ be given.
   The decomposition $u = v^\circ + w^\circ$, where $w^\circ = Q^\circ u$, is stable in the sense that
   \[
      v^{\circ\T} \D v^\circ + w^{\circ\T} \L w^\circ \lesssim u^\T \L u,
   \]
   where the implied constant is independent of the penalty parameter $\eta$.
\end{lem}
\begin{proof}
   Note that for any interior edge $(i,j) \in \E^\circ$, the associated weight $w_{ij}$ is independent of $\eta$, and so
   \[
      \sum_{(i,j) \in \E^\circ} w_{ij} (w_i - w_j)^2
         \lesssim \sum_{(i,j) \in \E^\circ} w_{ij} (u_i - u_j)^2,
   \]
   where the above expressions have no dependence of $\eta$.
   Since the subgraph $\G^\circ \subseteq \G$ consists of those vertices that belong to no boundary edges, we have that $w_i = u_i$ for any $i \in \V^\partial = \V \setminus \V^\circ$.
   Therefore, writing $\E = \E^\circ \cup E^\partial$,
   \begin{align*}
      w^{\circ\T} \L w^\circ =
      \sum_{\e_{ij} \in \E} w_{ij} (w_i^\circ - w_j^\circ)^2
         &= \sum_{\e_{ij} \in \E^\partial} w_{ij} (u_i - u_j)^2
         + \sum_{\e_{ij} \in \E^\circ} w_{ij} (w_i^\circ - w_j^\circ)^2 \\
         &\lesssim \sum_{\e_{ij} \in \E} w_{ij} (u_i - u_j)^2 = u^\T \L u.
   \end{align*}
   Similarly, the diagonal entries $w_{ii}$ associated with the vertices $V^\circ$ are independent of $\eta$, and $v_i = 0$ for any $i \in \V^\partial$, so $v^\T \D v \lesssim u^\T \L u$, and the conclusion follows.
\end{proof}

\begin{rem}
   The matrix associated with the subgraph $\G^\circ$ is an M-matrix whose entries do not depend on the penalty parameter $\eta$.
   Therefore, the standard algebraic multigrid theory for M-matrices applies, and so the decomposition of \Cref{lem:interior-decomposition} is expected to be stable not only with respect to the penalty parameter $\eta$, but also other relevant discretization parameters such as mesh size and coefficients.
\end{rem}

\begin{lem} \label{lem:boundary-projection}
   Let $u$ be given, and let $w^{(i)} = Q^{(i)} u$.
   Then, for all $j$,
   \[
      (w_j^{(i)} - u_j)^2 \lesssim \sum_{k \in \V^{(i)}} (u_k - u_j)^2.
   \]
\end{lem}
\begin{proof}
   Recalling the properties of the operator $Q^{(i)}$,
   \[
      \begin{aligned}
      (w_j^{(i)} - u_j)^2
      &= \left( \sum_{k\in V^{(i)}} Q^{(i)}_{jk} u_k - u_j \right)^2
      = \left( \sum_{k\in V^{(i)}} Q^{(i)}_{jk} ( u_k - u_j ) \right)^2 \quad&& \text{by \ref{item:row-sum}} \\
      &\lesssim \sum_{k\in V^{(i)}} \left( Q^{(i)}_{jk} ( u_k - u_j ) \right)^2
      \leq \sum_{k\in V^{(i)}} ( u_k - u_j ) ^2 \quad&& \text{by \ref{item:l2-norm}.}
      \end{aligned}
   \]
\end{proof}

\begin{lem} \label{lem:boundary-decomposition}
   Let $Q^\partial = \prod_i Q^{(i)}$, and let $u$ be given.
   Then, the decomposition $u = v + w$, where $w = Q^\partial u$ is stable in the sense that
   \[
      v^\T \D v + w^\T \L w \lesssim u^\T \L u,
   \]
   where the implied constant is independent of the penalty parameter $\eta$.
\end{lem}
\begin{proof}
   Using the result of \Cref{lem:boundary-projection}
   \begin{align*}
      w^\T \L w
         &= \sum_{(i,j) \in \E} w_{ij} (w_i - w_j)^2
         \lesssim \sum_{(i,j) \in \E} w_{ij} \left(
            (u_i - u_j)^2 + \left( w_i - u_i \right)^2 + \left( w_j - u_j \right)^2
         \right)\\
         &\lesssim \sum_{(i,j) \in \E} w_{ij} \left(
               (u_i - u_j)^2
               + \sum_{k \in \V^{(i)}} (u_k - u_i)^2
               + \sum_{k \in \V^{(j)}} (u_k - u_j)^2
            \right)
   \end{align*}
   The number of vertices in each of the subgraphs $\G^{(i)}$ is $\mathcal{O}(1)$, so
   \[
      \sum_{(i,j) \in \E} w_{ij} \sum_{k \in \V^{(i)}} (u_k - u_i)^2
      \lesssim \sum_{(i,j) \in \E} w_{ij} (u_i - u_j)^2,
   \]
   and we can conclude that
   \[
      w^\T \L w \lesssim \sum_{(i,j) \in \E} w_{ij} (u_i - u_j)^2 = u^\T \L u.
   \]
   We now turn our attention to the term $v^\T \D v$.
   Note that
   \[
      v^\T \D v = \sum_i w_i v_i^2 = \sum_i w_i (w_i - u_i)^2.
   \]
   By \Cref{lem:boundary-projection},
   \[
      \sum_i w_i (w_i - u_i)^2
      \lesssim \sum_i w_i \sum_{k \in \V^{(i)}} (u_k - u_i)^2. \myqed
   \]
\end{proof}

\begin{thm} \label{thm:decomposition}
   Let $u$ be given.
   Then, the decomposition
   \begin{equation}
      u = v + w \qquad\text{where}\qquad w = Q u
   \end{equation}
   is stable in the sense that
   \[
      v^\T \D v + w^\T \L w \lesssim u^\T \L u,
   \]
   where the implied constant is independent of $\eta$.
\end{thm}
\begin{proof}
   The stability of this decomposition follows from the stability of the interior and boundary decompositions, demonstrated above.
   First, consider the decomposition $u = w^\circ + v^\circ$, where $w^\circ = Q^\circ u$.
   By \Cref{lem:interior-decomposition}, this decomposition is stable, i.e.
   \begin{equation} \label{eq:interior-decomp}
      v^{\circ\T} \D v^\circ + w^{\circ\T} \L w^\circ \lesssim u^\T \L u.
   \end{equation}
   Then, further decompose $w^\circ = w + v^\partial$, where $w = Q^\partial w^\circ$, and define $v$ by $v = v^\partial + v^\circ$.
   By \Cref{lem:boundary-decomposition}, this decomposition is stable in the sense that
   \begin{equation} \label{eq:boundary-decomp}
      v^{\partial\T} \D v^\partial + w \L w \lesssim w^{\circ\T} \L w^\circ.
   \end{equation}
   Therefore,
   \begin{equation*}
      \begin{aligned}
         v^\T \D v + w^\T \L w
         &= (v^\partial + v^\circ)^\T \D (v^\partial + v^\circ) + w^\T \L w \\
         &\lesssim v^{\partial\T} \D v^\partial + v^{\circ\T} \D v^\circ + w^\T \L w \\
         &\lesssim v^{\circ\T} \D v^\circ + w^{\circ\T} \L w^\circ && \text{by \eqref{eq:boundary-decomp}}\\
         &\lesssim u^\T \L u && \text{by \eqref{eq:interior-decomp},}
      \end{aligned}
   \end{equation*}
   and the result follows.
\end{proof}

As a consequence of \Cref{thm:decomposition}, classical algebraic multigrid applied to the low-order DG discretization described above will result in uniform convergence, independent of the DG penalty parameter $\eta$.
Standard AMG theory regarding the robustness of the convergence for M-matrices with respect to discretization parameters such as mesh size and variation in the coefficients also carries over to this case.
This convergence theory is verified numerically in \Cref{sec:dg-results}.

\begin{rem}[CG preconditioning for DG methods]
   The use of continuous Galerkin discretizations (together with a smoothing operation, such as Jacobi or Gauss--Seidel) as preconditioners for discontinuous Galerkin methods has been studied extensively in the literature \cite{Dobrev2006,Antonietti2016,Pazner2020a,Pazner2021b}.
   \Cref{thm:decomposition} gives an alternative, elementary, proof of the optimality of CG preconditioning for DG methods.
   Each subgraph $\G^{(i)}$ corresponds to a ``duplicated'' DG degree of freedom in the high-order problem, and hence also maps to a single $H^1$ degree of freedom.
   Note that defining the operators $Q^{(i)}$ as the Oswald averaging operators (mapping $u_i$ to $(\# \V^{(i)})^{-1} \sum u_i$, cf.~\cite{Burman2007,Oswald1993}), properties \ref{item:l2-norm} and \ref{item:row-sum} are satisfied.
   This can be viewed as using the $H^1$-conforming subspace of the DG finite element space as an AMG coarse space.
   The conclusions of \Cref{thm:decomposition} hold for this case, showing that the resulting two-level method with Jacobi smoothing results in uniform convergence.
\end{rem}

\section{Numerical results} \label{sec:numerical-results}

The high-order and low-order-refined discretizations described in this work were implemented in the MFEM open-source finite element software library (cf.\ \cite{Anderson2020}, \url{https://mfem.org}).
The only modification made to MFEM's existing $\Hcurl$ and $\Hdiv$ discretizations was the implementation of the histopolation basis; these basis functions can be computed in a straightforward manner using partial sums of the derivatives of the standard Lagrange basis functions, as described in \Cref{rem:basis}.
The high-order finite element operators were constructed using \textit{partial assembly}, such that the action of the operator is applied without assembling the corresponding matrix, using geometric factors and coefficients that are precomputed at quadrature points.
Basis function evaluation and numerical integration are performed using sum factorization.
Per degree of freedom, this technique requires $\mathcal{O}(1)$ storage and $\mathcal{O}(p)$ operations.

In the MFEM library, the high-order operator can be assembled using the interpolation--histopolation basis by choosing the basis types \texttt{BasisType::GaussLobatto} (interpolation at the Gauss--Lobatto nodes) and \texttt{BasisType::IntegratedGLL} (histopolation using the Gauss--Lobatto subcells).
The low-order-refined versions of the high-order operators can be constructed in one line of code using the \texttt{LORDiscretization} class.
Preconditioners for the high-order discretization based on the low-order-refined matrices can similarly be constructed in one line using the \texttt{LORSolver} class.
The \texttt{lor\_solvers} miniapp, and its parallel counterpart \texttt{plor\_solvers}, illustrate the construction of low-order-refined discretizations and solvers, and come distributed with MFEM's source code, available at \url{https://github.com/mfem/mfem}.

The low-order refined mass and stiffness matrices $M_h$ and $K_h$ in $\Hcurl$ and $\Hdiv$ correspond to the standard finite element discretizations using lowest-order \Nedelec and Raviart--Thomas elements, posed on a refined mesh.
While any effective preconditioner can be used for the resulting low-order system, in this work we mainly make use of the algebraic multigrid preconditioners available in \textit{hypre} \cite{Falgout2002}.
In particular, the Auxiliary-Space Maxwell (AMS) solver (cf.\ \cite{Kolev2009}) is used for $\Hcurl$ problems, and the Auxiliary-Space Divergence (ADS) solver (cf.\ \cite{Kolev2012}) is used for $\Hdiv$ problems.
Classical algebraic multigrid is used for the DG discretizations.
A key feature of the LOR preconditioning approach is that the resulting high-order solvers inherit performance benefits and scalability from the traditional low-order solver implementations.
For example, since \textit{hypre}'s AMG solvers are highly scalable on massively parallel supercomputers, and also feature GPU acceleration, the LOR-based preconditioners also enjoy favorable scalability and GPU acceleration.

\subsection{Interpolation and histopolation equivalences} \label{sec:equiv-constants}

The spectral equivalence results in this paper are consequences of the one-dimensional norm equivalences of the interpolation and histopolation operators, cf.\ \Cref{prop:1d-l2-equiv}.
In this section, we numerically estimate the constants of the $L^2$ norm equivalences of the one-dimensional interpolation and histopolation operators, $\mathcal{I}_p$, $\mathcal{I}_h$, $\mathcal{H}_{p-1}$, and $\mathcal{H}_h$ defined on Gauss--Lobatto nodes.
For the interpolation operators, we will also consider the \textit{numerically integrated} $L^2$ norm, $\| \cdot \|_{0,\NI}$, which is computed using $p+1$ (collocated) Gauss--Lobatto quadrature points.
Note that the exactly integrated $L^2$ norm $\| \cdot \|_0$ and its numerically integrated counterpart $\| \cdot \|_{0,\NI}$ are equivalent, independent of polynomial degree $p$ \cite{Canuto1982}.
The use of numerical integration (inexact quadrature) can decrease the condition number of the preconditioned system; this effect has been studied in \cite{Fischer1997,Canuto2010,Bello-Maldonado2019}.
We numerically evaluate the value of the constants in estimates of the form
\[
   c \| \mathcal{I}_h(\mathsf{u}) \|_0 \leq \| \mathcal{I}_p(\mathsf{u}) \|_0 \leq C \| \mathcal{I}_h(\mathsf{u}) \|_0.
\]
The quantity $Cc^{-1}$ is shown for polynomial degrees $2 \leq p \leq 64$ in \Cref{fig:interp-histop}.
As expected given \Cref{prop:1d-l2-equiv}, the quantity $Cc^{-1}$ remains asymptotically bounded, independent of the polynomial degree $p$.
Furthermore, the constants corresponding to the numerically integrated norms $\| \cdot \|_{0,\NI}$ are smaller than those corresponding to the fully integrated $L^2$ norm.
This indicates that it is beneficial to use collocated quadrature when assembling the low-order system.
The use of collocated quadrature for the high-order system will also lead to a better conditioned systems, however, this will result in a modified discretization that may not be desired.
These one-dimensional constants can be used to estimate the constants of the 2D and 3D equivalences, including for the gradient, curl, and divergence operators, using the results of \Cref{prop:grad-equiv,prop:curl-equiv,,prop:div-equiv}. 

\begin{figure}
   \centering
   \includegraphics{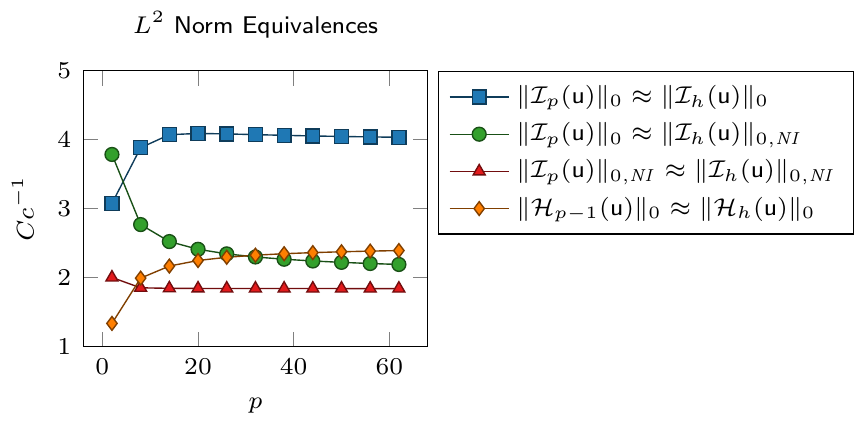}

   \caption{Constants of $L^2$ norm equivalence for the one-dimensional interpolation and histopolation operators.}
   \label{fig:interp-histop}
\end{figure}

\subsection{Single element condition numbers}

In this section, we compute condition numbers of the preconditioned mass and stiffness matrices on the reference elements in 2D and 3D, $\Omega = \refinterval^d$.
The linear system is given by the sum of the mass and stiffness matrices, $A_p = M_p + K_p$, where $M_p$ and $K_p$ are the mass and stiffness matrices corresponding to one of the finite element spaces $\Vp, \Wp, \Xp, \Zp$.
Let $A_h$ denote the corresponding low-order-refined system.
Dirichlet boundary conditions are enforced at the domain boundary.
The high-order system is integrated with $(p+1)^d$ Gauss--Lobatto quadrature points (collocated quadrature, or ``numerical integration\rlap{,}'' cf.\ \cite{Canuto2010}); this typically results in better conditioned systems, but the condition numbers still remain asymptotically bounded in the case of exact integration \cite{Fischer1997,Bello-Maldonado2019}.
The condition number of the matrix $A_h^{-1} A_p$ is reported in \Cref{fig:element-conditioning}.
Note than in 2D, the $\Hcurl$ and $\Hdiv$ spaces coincide, and so the reported condition numbers are identical.
The low-order preconditioner for the $H^1$ system results in a condition number that is bounded by $\pi^2/4$ in all cases; the bound of $\pi^2/4$ was first established for the case of low-order preconditioning of spectral methods in \cite{Haldenwang1984}.

\begin{figure}
   \centering
   \includegraphics{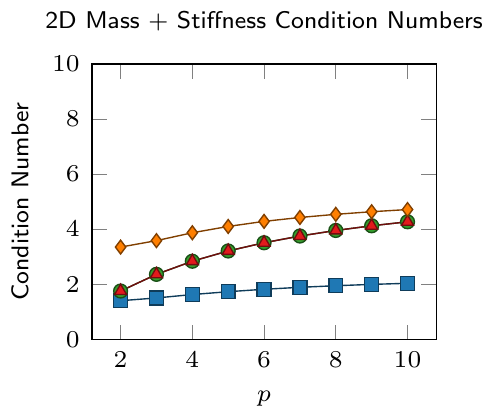}
   \includegraphics{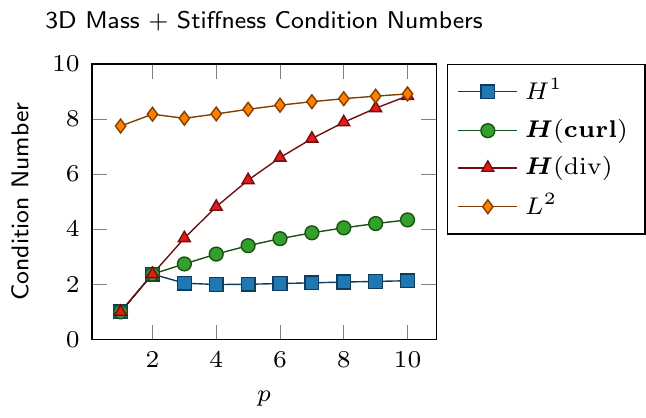}
   \caption{Condition numbers of the matrices $A_h^{-1} A_p$, where $A_p = M_p + K_p$ on the reference element $\refinterval^d$ in 2D and 3D.}
   \label{fig:element-conditioning}
\end{figure}

It is also possible to bound the single-element condition numbers given the norm equivalence constants computed in \Cref{sec:equiv-constants}.
Let $\kappa_\mathcal{I}$ denote the $L^2$ norm equivalence constant $Cc^{-1}$ associated with the interpolation operators, and similarly for $\kappa_\mathcal{H}$.
For example, in two dimensions, the high-order and low-order discrete gradient operators are connected through the relation
\[
   \mathcal{G}^2_h =
   \left(\begin{array}{cc}
      \mathcal{H}_h \mathcal{H}_{p-1}^{-1} \otimes \mathcal{I}_h \mathcal{I}_{p}^{-1}& 0 \\
      0 & \mathcal{I}_h \mathcal{I}_{p}^{-1} \otimes \mathcal{H}_h \mathcal{H}_{p-1}^{-1}
   \end{array}
   \right) \mathcal{G}^2_p.
\]
Therefore, the resulting condition number can be bounded by $\kappa_\mathcal{I} \kappa_\mathcal{H}$.
Estimates for the remaining spaces and operators can be derived similarly (with the exception of the DG interior penalty stiffness matrix, whose analysis, described in \Cref{sec:dg}, requires a different framework).
A comparison of the estimated and computed condition numbers is shown in \Cref{tab:estimate-comparison}.
In all of the cases, the estimates computed using products of the one-dimensional constants give an upper bound for the computed condition numbers.
In 2D and 3D, the estimates of the form $\kappa_\mathcal{H}^d$ are quite sharp; the estimates for the $H^1$ and $\Hcurl$ cases are more pessimistic.

\begin{table}
   \centering
   \caption{Comparison of estimated and computed condition numbers for the single element case.}
   \label{tab:estimate-comparison}

   \begin{tabular}{c|cc|cc}
      \multicolumn{5}{c}{2D Case}\\
      \toprule
      $p$ & $M_{\Vp} + K_{\Vp}$ & $\kappa_\mathcal{I}\kappa_\mathcal{H}$ & $M_{\Wp} + K_{\Wp}$ & $\kappa_\mathcal{H}^2$ \\
      \midrule
      2 & 1.41 & 2.67 & 1.76 & 1.78 \\
      4 & 1.63 & 3.18 & 2.84 & 2.86 \\
      6 & 1.82 & 3.49 & 3.51 & 3.52 \\
      8 & 1.95 & 3.68 & 3.95 & 3.96 \\
      10 & 2.04 & 3.82 & 4.27 & 4.28 \\
      \bottomrule
   \end{tabular}

   \vspace{\floatsep}

   \begin{tabular}{c|cc|cc|cc}
      \multicolumn{7}{c}{3D Case}\\
      \toprule
      $p$ & $M_{\Vp} + K_{\Vp}$ & $\kappa_\mathcal{I}^2\kappa_\mathcal{H}$ & $M_{\Wp} + K_{\Wp}$ & $\kappa_\mathcal{I}\kappa_\mathcal{H}^2$ & $M_{\Xp} + K_{\Xp}$ & $\kappa_\mathcal{H}^3$ \\
      \midrule
      2 & 2.37 & 5.33 & 2.37 & 3.56 & 2.37 & 2.37 \\
      4 & 1.99 & 5.98 & 3.10 & 5.37 & 4.82 & 4.83 \\
      6 & 2.03 & 6.48 & 3.66 & 6.54 & 6.60 & 6.61 \\
      8 & 2.08 & 6.81 & 4.05 & 7.33 & 7.88 & 7.89 \\
      10 & 2.13 & 7.05 & 4.34 & 7.89 & 8.84 & 8.84 \\
      \bottomrule
   \end{tabular}
\end{table}

\subsection{Mass matrix preconditioning} \label{sec:mass-results}

In this section, we numerically compare several options for preconditioning the high-order mass matrix.
In particular, given their simplicity, efficiency, and effectiveness for mass matrix problems, we focus on diagonal preconditioning approaches.
In light of \Cref{sec:mass-precond}, all of the diagonal preconditioners considered are spectrally equivalent to the high-order mass matrix, independent of $p$, and so the results in this section represent a numerical and empirical comparison of the constants of equivalence.

For the $H^1$ mass matrix $M_{\Vp}$, we compare Jacobi preconditioning using the nodal Gauss--Lobatto basis, which we denote ``Jacobi (Lobatto)\rlap{,}'' to ``LOR Jacobi'', which indicates using the diagonal of the low-order mass matrix $M_{\Vh}$ as a preconditioner.
For the spaces $\Wp, \Xp$, and $\Yp$, we also consider Gauss--Legendre and histopolation bases for the components for which continuity is not enforced.
The corresponding diagonal preconditioners are denoted ``Jacobi (Legendre)'' and ``Jacobi (Integrated)\rlap{,}'' respectively.
Note that the diagonal of the high-order mass matrix can be constructed efficiently in $\mathcal{O}(p^d)$ operations, without assembling the entire matrix (see, e.g.~\cite{Ronquist1987}).
Therefore, all of the preconditioning options considered in this section are suitable for the matrix-free context.
All of these options should result in uniformly well-conditioned systems, independent of the polynomial degree $p$, cf.~\Cref{sec:mass-precond}, however the constants of equivalence will be different for each of the choices.

We consider two 3D meshes: a simple Cartesian grid, and a fully unstructured hexahedral mesh (including skewed elements that are not given by affine transformations of the unit cube).
For each of the spaces $\Vp, \Wp, \Xp, \Ypm$, we iteratively solve the linear system for the high-order mass matrix with a random right-hand side to a relative tolerance of $10^{-12}$.
The iteration counts are shown in \Cref{fig:mass}.
In general, the iteration counts are larger for the unstructured mesh than for the structured grid.
Many of the preconditioners display a relatively mild preasymptotic increase in iterations with increasing $p$.
The Jacobi preconditioner with Gauss--Legendre basis typically gives rise to the smallest number of iterations.
Note that the $L^2$ mass matrix $M_{\Ypm}$ on affine elements is exactly integrated using Gauss--Legendre quadrature, and hence ``Jacobi (Legendre)'' is actually an exact solver in this case (convergence is always attained in only one iteration).

\begin{figure}
   \centering
   \begin{tabular}{ll}
      \includegraphics{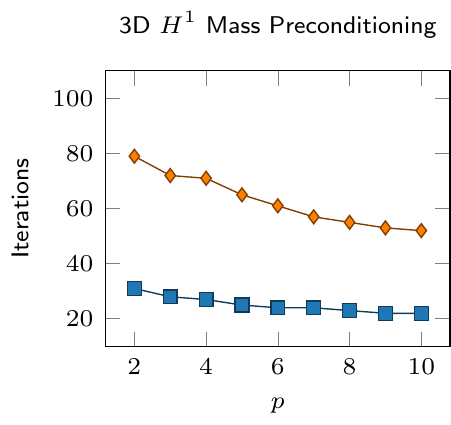} &
      \includegraphics{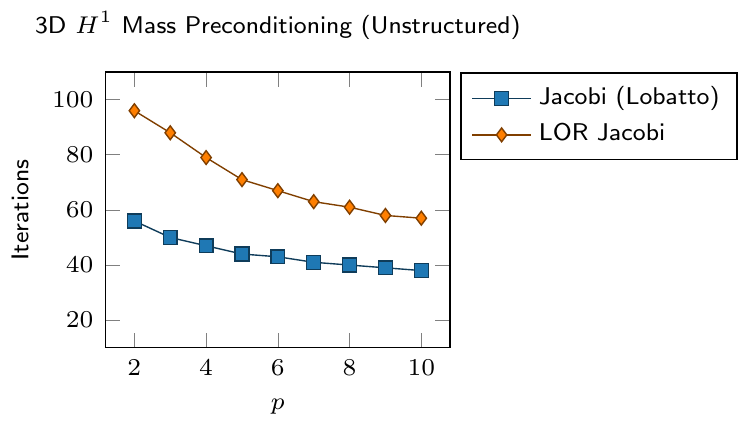} \\
      \includegraphics{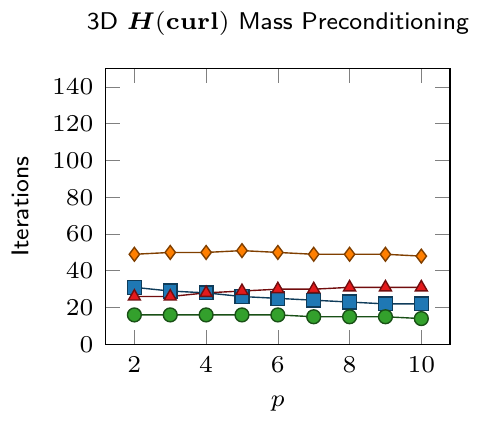} &
      \includegraphics{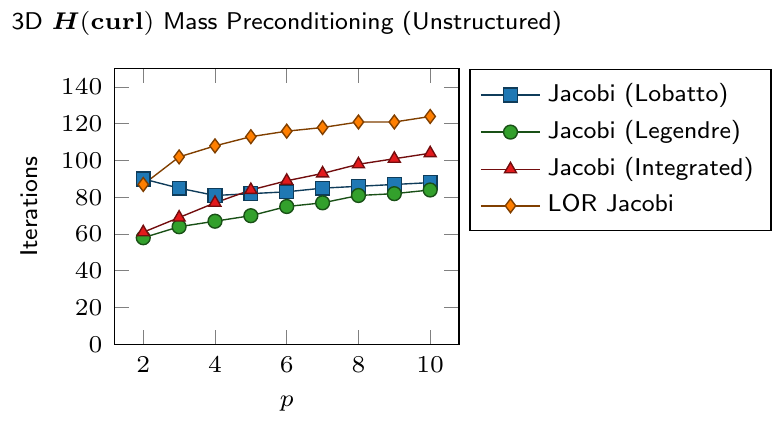} \\
      \includegraphics{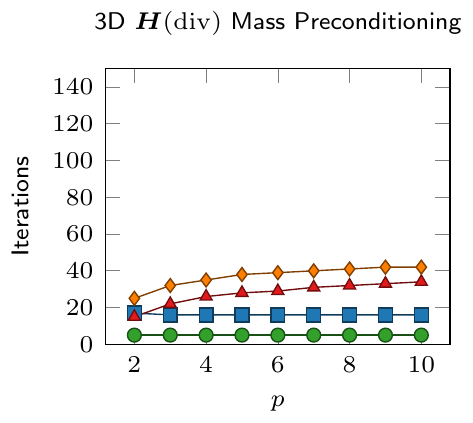} &
      \includegraphics{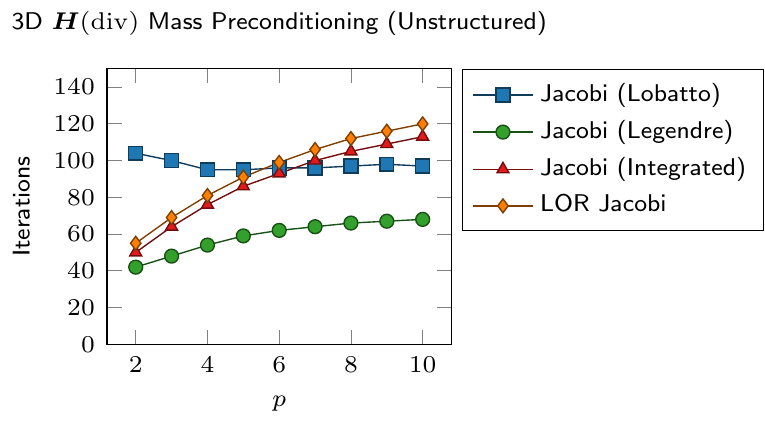} \\
      \includegraphics{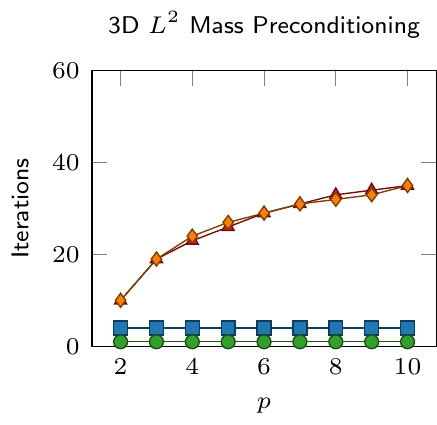} &
      \includegraphics{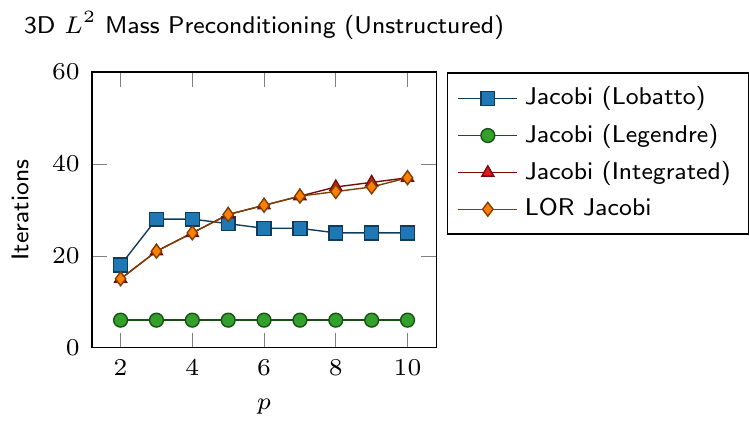}
   \end{tabular}
   \caption{
      Comparisons of conjugate gradient iterations for different diagonal preconditioning strategies for the high-order mass matrix.
   }
   \label{fig:mass}
\end{figure}

\subsection{Definite Maxwell problem: copper wire}

In this section, we consider the simulation of electromagnetic diffusion of a copper wire in air, cf.\ \cite{Kolev2009}.
We solve the definite Maxwell problem
\[
   \nabla \times \nabla \times \bm u + \beta \bm u = \bm f,
\]
where $\beta$ represents the conductivity coefficient.
This coefficient is given by a piecewise constant, with $\beta_{\text{air}} = 10^{-6}$ and $\beta_{\text{copper}} = 1$.
This problem is solved on a mesh with 21{,}060 curved $\mathcal{Q}_3$ elements.
A schematic of this problem and the computational mesh are shown in \Cref{fig:copper-wire}.
The right hand side is chosen to be $\bm f = 1$.

\begin{figure}
   \centering
   \includegraphics{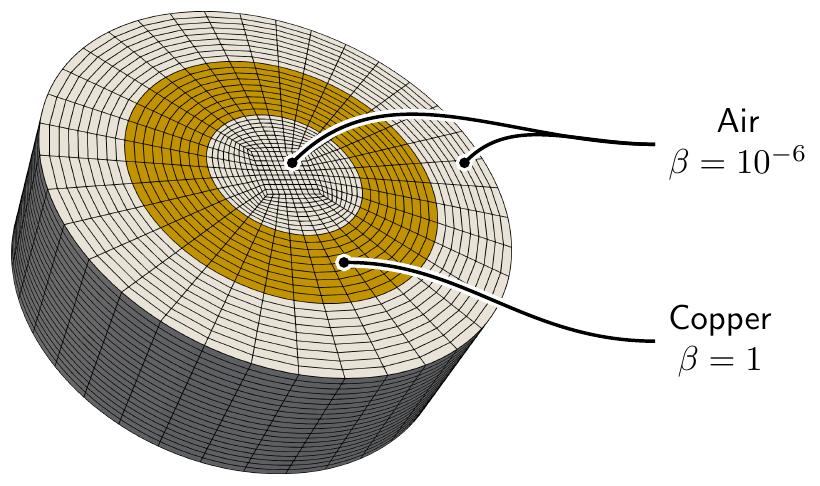}
   \caption{
      Mesh of the copper wire Maxwell problem.
      The conductivity coefficient is a piecewise constant coefficient determined by the material of the elements.
   }
   \label{fig:copper-wire}
\end{figure}

We compare the number of iterations and computational time required to solve this problem to a relative tolerance of $10^{-12}$ using the auxiliary space Maxwell algebraic multigrid solver, applied directly to the assembled high-order system (denoted ``Matrix-Based AMS''), and applied to the low-order refined system (``LOR--AMS'').
This problem is solved using 144 MPI ranks of LLNL's \textit{Quartz} supercomputer.
The results are shown in \Cref{tab:copper-wire}.
The iterations required for the LOR--AMS solver require at most $1.5\times$ as many iterations as the matrix-based AMS solver.
However, the assembly time (which for the LOR--AMS solver denotes the time required to assemble the LOR matrix, as well as the time required for the ``partial assembly'' of the high-order operator) is significantly reduced for the LOR--AMS solver.
Additionally, the number of nonzero entries of the system matrix (and hence the memory requirements for the solver) are significantly reduced for the LOR--AMS solver.
Speedup and memory reduction factors are reported in \Cref{tab:copper-wire-speedup}.
For polynomial degree $p=6$, the total runtime is reduced by a factor of $25$, and the memory usage is reduced by a factor of $35$.

\begin{table}
   \centering
   \caption{Convergence results and runtimes for the copper wire Maxwell problem.}
   \label{tab:copper-wire}
   \setlength{\tabcolsep}{10pt}
   \begin{tabular}{ccSSSS[table-number-alignment=right,table-column-width=0.9in]l}
      \toprule
      \multicolumn{7}{c}{LOR--AMS} \\
      \midrule
      $p$ & Its. & \multicolumn{1}{c}{Assembly (s)} & \multicolumn{1}{c}{AMG Setup (s)} & \multicolumn{1}{c}{Solve (s)} & \multicolumn{1}{c}{\# DOFs} & \multicolumn{1}{c}{\# NNZ} \\
      \midrule
          2 &         41 & 0.082 & 0.277 & 0.768 & 516820 & $1.65 \times 10^{7}$ \\
          3 &         63 & 0.251 & 0.512 & 2.754 & 1731408 & $5.64 \times 10^{7}$ \\
          4 &         75 & 0.679 & 1.133 & 7.304 & 4088888 & $1.34 \times 10^{8}$ \\
          5 &         62 & 1.574 & 2.185 & 11.783 & 7968340 & $2.61 \times 10^{8}$ \\
          6 &         89 & 3.336 & 4.024 & 30.702 & 13748844 & $4.51 \times 10^{8}$ \\
      \midrule
      \multicolumn{7}{c}{Matrix-Based AMS} \\
      \midrule
      $p$ & Its. & \multicolumn{1}{c}{Assembly (s)} & \multicolumn{1}{c}{AMG Setup (s)} & \multicolumn{1}{c}{Solve (s)} & \multicolumn{1}{c}{\# DOFs} & \multicolumn{1}{c}{\# NNZ} \\
      \midrule
      2 & 39 & 0.140 & 0.385 & 1.423 & 516820 & $5.24 \times 10^{7}$ \\
      3 & 44 & 1.368 & 1.572 & 9.723 & 1731408 & $4.01 \times 10^{8}$ \\
      4 & 49 & 9.668 & 5.824 & 45.277 & 4088888 & $1.80 \times 10^{9}$ \\
      5 & 53 & 61.726 & 15.695 & 148.757 & 7968340 & $5.92 \times 10^{9}$ \\
      6 & 56 & 502.607 & 40.128 & 424.100 & 13748844 & $1.59 \times 10^{10}$ \\
      \bottomrule
  \end{tabular}
\end{table}

\begin{table}
   \centering
   \caption{
      Speedup and memory reduction for the copper wire Maxwell problem.
      Runtime includes assembly, solver setup, and solve times.
      Memory indicates the size of the assembled system matrix in CSR format.
   }
   \label{tab:copper-wire-speedup}
   \begin{tabular}{c|SS|SS|SS}
      \toprule
      & \multicolumn{2}{c|}{LOR--AMS} & \multicolumn{2}{c|}{Matrix-Based AMS} \\
      $p$ & {Runtime (s)} & {Memory (GB)} & {Runtime (s)} & {Memory (GB)} & {Speedup} & {Memory Reduction} \\
      \midrule
          2 &  1.13 &  0.19 &  1.95 &  0.59 &  1.73$\times$ &  3.16$\times$ \\
          3 &  3.52 &  0.64 & 12.66 &  4.49 &  3.60$\times$ &  7.05$\times$ \\
          4 &  9.11 &  1.51 & 60.77 & 20.09 &  6.67$\times$ & 13.31$\times$ \\
          5 & 15.54 &  2.95 & 226.18 & 66.15 & 14.55$\times$ & 22.45$\times$ \\
          6 & 38.06 &  5.09 & 966.83 & 178.18 & 25.40$\times$ & 35.00$\times$ \\
      \bottomrule
  \end{tabular}
\end{table}

\subsection{Grad-div problem: crooked pipe}

In this section we consider the ``crooked pipe'' grad-div problem, which is a benchmark problem related to radiation diffusion simulations \cite{Graziani2000,Gentile2001}.
The problem is posed on a cylindrical sector, consisting of two material subdomains.
The mesh elements near the interface between the subdomains are refined anisotropically, leading to highly stretched elements.
The mesh used for this problem is shown in \Cref{fig:crooked-pipe}.
We solve the problem
\[
   \nabla \left( \alpha \nabla \cdot \bm u \right) - \beta \bm u = \bm f,
\]
where the coefficients $\alpha$ and $\beta$ are given piecewise constant values according to the materials.
In the larger subregion (colored green in \Cref{fig:crooked-pipe}), we take $\alpha = 1.88 \times 10^{-3}$ and $\beta = 2000$.
In the smaller subregion (colored blue in \Cref{fig:crooked-pipe}), we take $\alpha = 1.641$ and $\beta = 0.2$.

\begin{figure}
   \centering
   \raisebox{-0.5\height}{\includegraphics{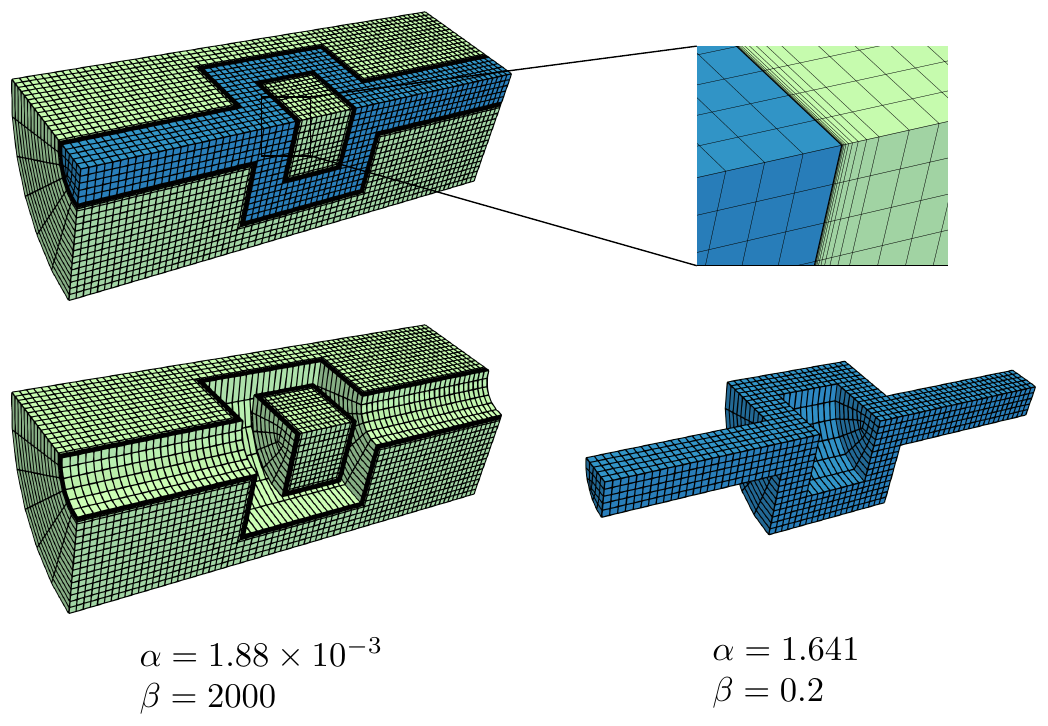}}
   \raisebox{-0.5\height}{\includegraphics[width=2in]{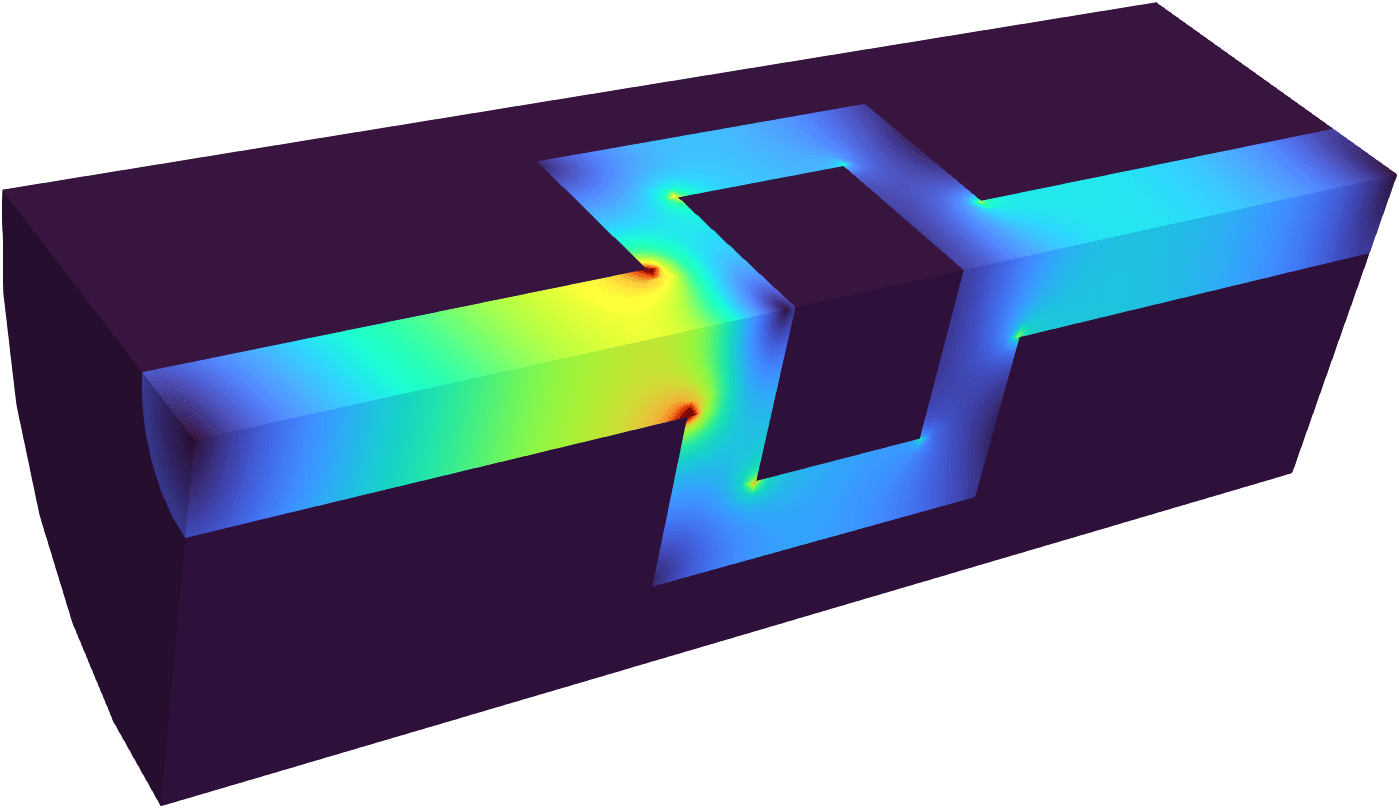}}
   \caption{
      Crooked pipe grad-div problem.
      The domain is split into two materials.
      The mesh is refined in the vicinity of the material interface, resulting in highly stretched anisotropic elements.
   }
   \label{fig:crooked-pipe}
\end{figure}

Analogously to the previous section, we compare the number of iterations and computational time required to solve this problem to a relative tolerance of $10^{-12}$ using the auxiliary space divergence algebraic multigrid solver, applied directly to the assembled high-order system (denoted ``Matrix-Based ADS''), and applied to the low-order refined system (``LOR--ADS'').
As in the previous section, this problem is solved in parallel using 144 MPI ranks.
These results are shown in \Cref{tab:crooked-pipe}.
For $p=2$, the LOR--ADS solver requires about $1.3\times$ as many iterations as the matrix-based ADS solver.
This factor increases to about $2.3\times$ for $p=6$.
This preasymptotic increase is largely consistent with the conditioning results shown in \Cref{fig:element-conditioning,fig:mass}.
Despite this increase in the number of iterations, the solve time alone is reduced by more than a factor of $10\times$ by using the LOR--ADS preconditioner for $p=6$.
Total speedup and memory reduction results are shown in \Cref{tab:crooked-pipe-speedup}.
For this test case, the reduction in memory usage is even more dramatic than for the $\Hcurl$ problem; at $p=6$, the memory required to store the system matrix is reduced by a factor of $77\times$ using the LOR method.

\begin{table}
   \centering
   \caption{Convergence results for the crooked pipe grad-div problem.}
   \label{tab:crooked-pipe}

   \setlength{\tabcolsep}{10pt}
   \begin{tabular}{ccSSSS[table-number-alignment=right,table-column-width=0.9in]l}
      \toprule
      \multicolumn{7}{c}{LOR--ADS} \\
      \midrule
      $p$ & Its. & \multicolumn{1}{c}{Assembly (s)} & \multicolumn{1}{c}{AMG Setup (s)} & \multicolumn{1}{c}{Solve (s)} & \multicolumn{1}{c}{\# DOFs} & \multicolumn{1}{c}{\# NNZ} \\
      \midrule
      2 & 115 & 0.027 & 0.299 & 1.782 & 356500 & $3.81 \times 10^{6}$ \\
      3 & 168 & 0.075 & 0.583 & 5.986 & 1190115 & $1.28 \times 10^{7}$ \\
      4 & 197 & 0.218 & 1.245 & 15.932 & 2805520 & $3.04 \times 10^{7}$ \\
      5 & 243 & 0.456 & 2.428 & 40.814 & 5461375 & $5.93 \times 10^{7}$ \\
      6 & 276 & 0.907 & 4.969 & 86.882 & 9416340 & $1.03 \times 10^{8}$ \\
      \midrule
      \multicolumn{7}{c}{Matrix-Based ADS} \\
      \midrule
      $p$ & Its. & \multicolumn{1}{c}{Assembly (s)} & \multicolumn{1}{c}{AMG Setup (s)} & \multicolumn{1}{c}{Solve (s)} & \multicolumn{1}{c}{\# DOFs} & \multicolumn{1}{c}{\# NNZ} \\
      \midrule
      2 & 86 & 0.032 & 0.459 & 2.486 & 356500 & $1.79 \times 10^{7}$ \\
      3 & 99 & 0.417 & 2.815 & 19.351 & 1190115 & $1.64 \times 10^{8}$ \\
      4 & 104 & 3.166 & 12.338 & 98.844 & 2805520 & $8.16 \times 10^{8}$ \\
      5 & 112 & 18.062 & 46.590 & 358.085 & 5461375 & $2.88 \times 10^{9}$ \\
      6 & 121 & 181.517 & 165.135 & 978.285 & 9416340 & $8.15 \times 10^{9}$ \\
      \bottomrule
  \end{tabular}
\end{table}

\begin{table}
   \centering
   \caption{
      Speedup and memory reduction for the crooked-pipe grad-div problem.
      Runtime includes assembly, solver setup, and solve times.
      Memory indicates the size of the assembled system matrix in CSR format.
   }
   \label{tab:crooked-pipe-speedup}
   \begin{tabular}{c|SS|SS|SS}
      \toprule
      & \multicolumn{2}{c|}{LOR--ADS} & \multicolumn{2}{c|}{Matrix-Based ADS} \\
      $p$ & {Runtime (s)} & {Memory (GB)} & {Runtime (s)} & {Memory (GB)} & {Speedup} & {Memory Reduction} \\
      \midrule
      2 &  2.11 &  0.04 &  2.98 &  0.20 &  1.41$\times$ &  4.60$\times$ \\
      3 &  6.64 &  0.15 & 22.58 &  1.84 &  3.40$\times$ & 12.44$\times$ \\
      4 & 17.40 &  0.35 & 114.35 &  9.13 &  6.57$\times$ & 26.07$\times$ \\
      5 & 43.70 &  0.68 & 422.74 & 32.21 &  9.67$\times$ & 47.11$\times$ \\
      6 & 92.76 &  1.18 & 1324.94 & 91.09 & 14.28$\times$ & 77.13$\times$ \\
      \bottomrule
  \end{tabular}
\end{table}

\subsection{Discontinuous Galerkin methods} \label{sec:dg-results}

In this section, we consider the low-order preconditioning for discontinuous Galerkin methods proposed in \Cref{sec:dg}.
As a test case, we use the solver benchmark problem proposed in \cite{Kolev2021}, and solve the constant-coefficient Poisson problem
\begin{align*}
   -\Delta u &= f \quad\text{in $\Omega$,} \\
   u &= 0 \quad\text{on $\partial\Omega$,}
\end{align*}
with homogeneous Dirichlet boundary conditions in the unit cube $\Omega = [0,1]^3$.
The right-hand side $f$ is determined by the prescribed exact solution $u$, which is given as the tensor-product of one-dimensional functions (parameterized by the so-called \textit{structure level} $n \in \mathbb{N}$), $u(x,y,z) = w_n(x) w_n (y) w_n(z)$.
The functions $w_n(x)$ are given by
\begin{align*}
   w_n(x) &= \sum_{j=0}^{n-1} u_{3^j}(x), \\
   u_k(x) &= \exp(-1 / s_k^2(x)) \operatorname{sign}(s_k(x)), \\
   s_k(x) &= \sin(2 j \pi x).
\end{align*}
We consider a family of \textit{Kershaw meshes} (cf.\ \cite{Kershaw1981}), parameterized by an anisotropy parameter $0 < \varepsilon \leq 1$.
These meshes are obtained by distorting a Cartesian grid, such that layers of elements with aspect ratio $1/\varepsilon$ are placed in opposing corners of the cube.
The mesh transitions through four intermediate layers in a ``Z'' pattern, giving rise to skewed and stretched elements.
When $\varepsilon = 1$, the mesh is a uniform Cartesian grid.
The geometric anisotropy induced by smaller values of $\varepsilon$ often proves challenging for linear solvers and preconditioners.
\Cref{fig:kershaw} illustrates examples of the Kershaw mesh for $\varepsilon = 1$ and $\varepsilon = 0.3$ on a mesh with $12^3$ elements.

\begin{figure}
   \centering
   \begin{tabular}{c@{\hskip 0.5in}c}
      \includegraphics[width=1.5in]{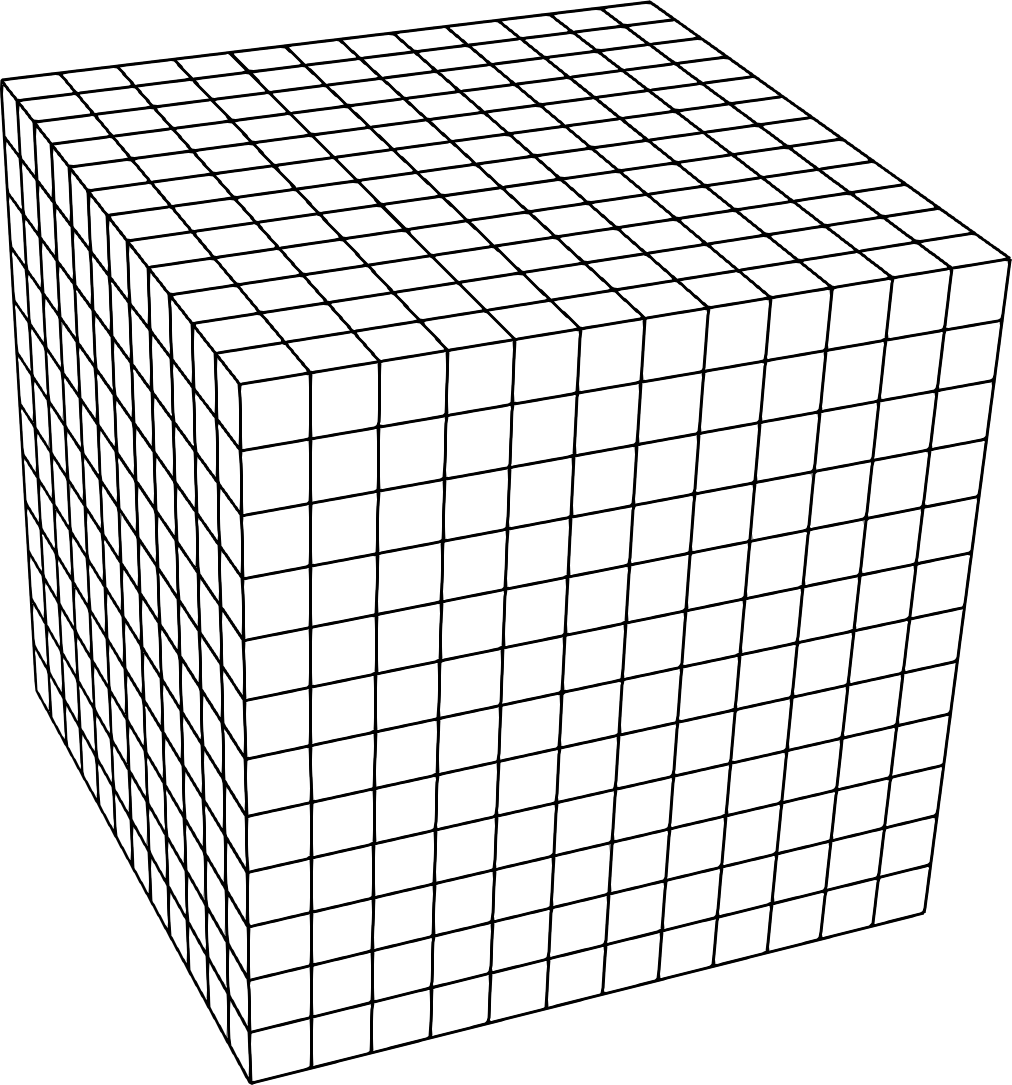} &
      \includegraphics[width=1.5in]{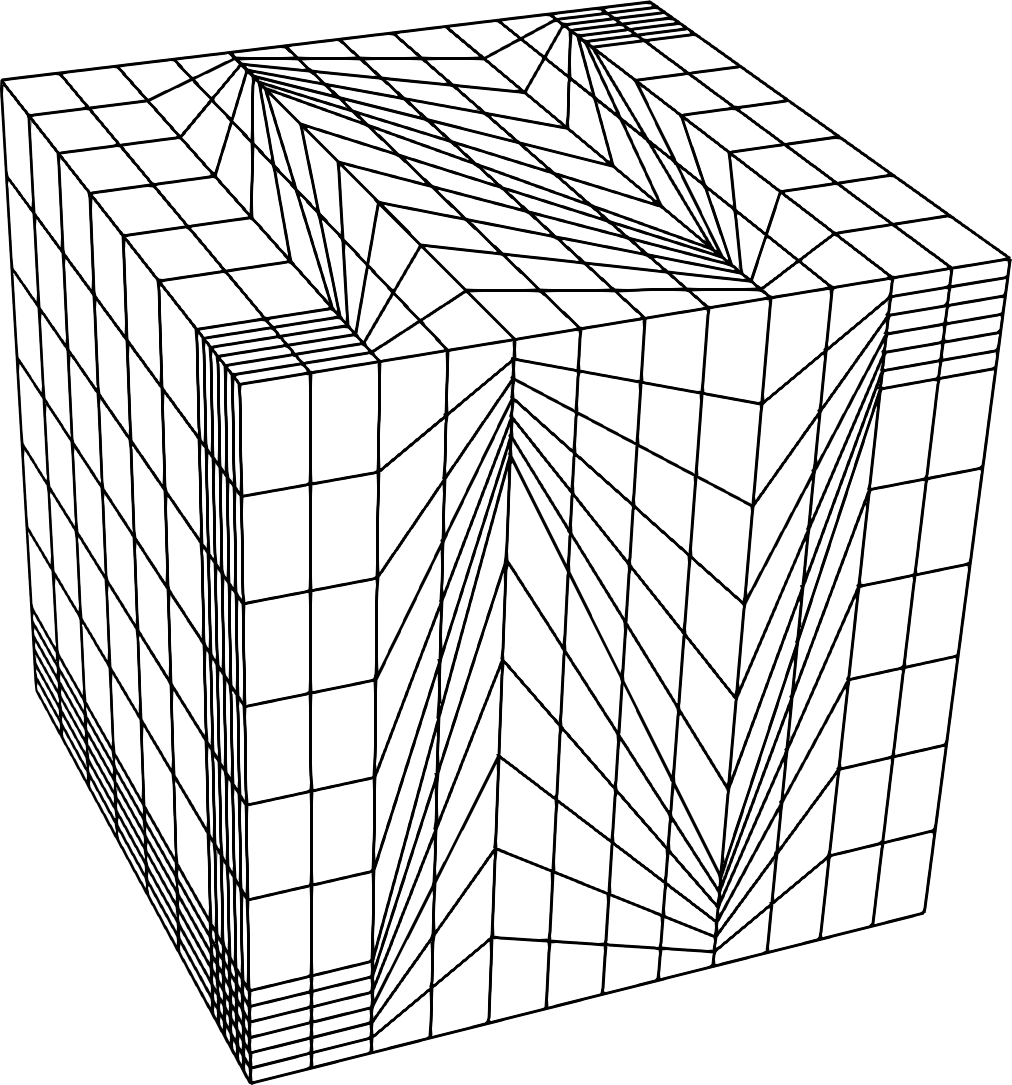} \\
      $\varepsilon = 1$ & $\varepsilon = 0.3$
   \end{tabular}
   \caption{Illustrations of the Kershaw mesh with $12^3$ elements and $\varepsilon = 1$ (left) and $\varepsilon = 0.3$ (right).}
   \label{fig:kershaw}
\end{figure}

We solve this problem using both $\varepsilon = 1$ (isotropic elements) and $\varepsilon = 0.3$ (anisotropic elements).
We begin with polynomial degree $p=1$ on a mesh with $36^3$ elements, corresponding to 373{,}248 degrees of freedom.
For increasing polynomial degrees $p=2,3,5,7,9$, the mesh is simultaneously coarsened to keep the total number of degrees of freedom fixed.
In each case, we assemble the corresponding low-order-refined system $K_{\Zh}$ using the piecewise constant DG discretization defined in \Cref{sec:dg}.
As a preconditioner, we use \textit{hypre}'s BoomerAMG algebraic multigrid with $\ell_1$--Jacobi smoothing \cite{Baker2011}.
In Table \ref{tab:dg-kershaw} we present the number of conjugate gradient iterations required to reduce the residual by a factor of $10^{12}$.
In addition to the iteration counts required to solve the high-order problem $K_{\Zp}$, we also present the number of AMG-preconditioned CG iterations required to solve the low-order-refined problem $K_{\Zh}$.
Furthermore, we compare these iteration counts to the number of iterations required to solve the high-order problem using ``CG--DG preconditioning\rlap{,}'' which is based on the idea of using the low-order-refined conforming (continuous Galerkin) problem as a preconditioner, together with a diagonal correction, cf.\ \cite{Dobrev2006,Antonietti2016,Pazner2020a}.

For $\varepsilon = 1$, the number of iterations required to solve the high-order problem using the low-order-refined problem (the column labeled $B_h K_{\Zp}$ in \Cref{tab:dg-kershaw}) as a preconditioner remains bounded, independent of $p$.
For $\varepsilon = 0.3$, the problem is more challenging because of the mesh-induced anisotropy, however, after a mild preasymptotic increase, the iteration counts appear to be uniform with respect to $p$.
These results corroborate the spectral equivalence demonstrated in \Cref{thm:dg-equivalence}.
In both of these cases, the iterations required to solve this problem using the low-order-refined preconditioner described in the present work are significantly less than those required to solve this problem using the CG--DG subspace preconditioning.
We note that for both $\varepsilon = 1$ and $\varepsilon = 0.3$, BoomerAMG applied to the low-order-refined problem (the column labeled $B_h K_{\Zh}$), the number of iterations required to converge remains bounded, independent of $p$, corroborating the conclusions of \Cref{thm:decomposition}.

\begin{table}
   \centering
   \caption{
      Convergence results for discontinuous Galerkin discretizations of the Poisson problem on Kershaw meshes with $\varepsilon = 1$ and $\varepsilon = 0.3$.
      Each column indicates the number of CG iterations required to converge to a relative tolerance of $10^{-12}$, where $K_{\Zp}$ and $K_{\Zh}$ denote the high-order and low-order-refined operators, respectively, $B_p$ denotes BoomerAMG formed using the high-order matrix, $B_h$ denotes BoomerAMG formed using the low-order-refined matrix, and CG--DG indicates the use of the conforming problem as a subspace correction preconditioner.
   }
   \label{tab:dg-kershaw}
   \begin{tabular}{c|ccc|ccc}
      \toprule
      & \multicolumn{3}{c|}{$\varepsilon=1$} & \multicolumn{3}{c}{$\varepsilon=0.3$} \\
      $p$ & $B_h K_{\Zp}$ & $B_h K_{\Zh}$ & CG--DG & $B_h K_{\Zp}$ & $B_h K_{\Zh}$ & CG--DG \\
      \midrule
          1 &    51 &        20 &       87 &      88 &          25 &        385 \\
          2 &    48 &        22 &       68 &      83 &          25 &        311 \\
          3 &    49 &        23 &       61 &      97 &          27 &        300 \\
          5 &    47 &        23 &       60 &     115 &          29 &        319 \\
          7 &    48 &        23 &       66 &     121 &          32 &        308 \\
          9 &    49 &        23 &       71 &     111 &          29 &        285 \\
      \bottomrule
   \end{tabular}
\end{table}

We additionally study the dependence of the convergence properties of the low-order-refined AMG preconditioners on the magnitude of the DG interior penalty parameter.
The condition number of the system $K_{\Zp}$ scales linearly with the penalty parameter $\eta$ (cf.\ \Cref{prop:dg-norm}), and geometric and algebraic multigrid preconditioners often give degraded convergence for large values of the penalty parameter.
As an example, we take the case of $p=1$ with $36^3$ elements, and compute the number of CG iterations required to converge to a relative tolerance of $10^{-12}$ for increasing values of the penalty parameter.
We compare BoomerAMG applied directly to $K_{\Zp}$ (this preconditioner is denoted $B_p$), BoomerAMG applied to $K_{\Zh}$ (this preconditioner is denoted $B_h$), and CG--DG preconditioning, and present the results in \Cref{fig:dg-penalty}.
We note that the iteration counts for BoomerAMG applied directly to $K_{\Zp}$ increase substantially as $\eta$ increases; for the case of $\varepsilon = 0.3$, the convergence criterion was not met in under 2000 iterations for $\eta \geq 10^3$.
The iteration counts for BoomerAMG formed using the low-order-refined system $K_{\Zh}$ and applied to both $K_{\Zp}$ and $K_{\Zh}$ remain bounded, independent of $\eta$, corroborating the results of \Cref{thm:dg-equivalence,thm:decomposition}.
Similarly, the iteration counts for the CG--DG preconditioner are bounded independent of $\eta$, cf.\ \cite{Antonietti2016,Pazner2020a}.

\begin{figure}
   \centering
   \includegraphics{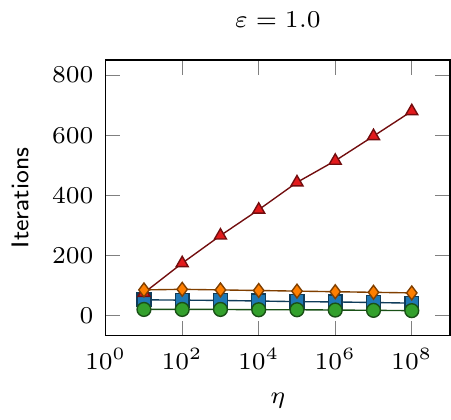}
   \includegraphics{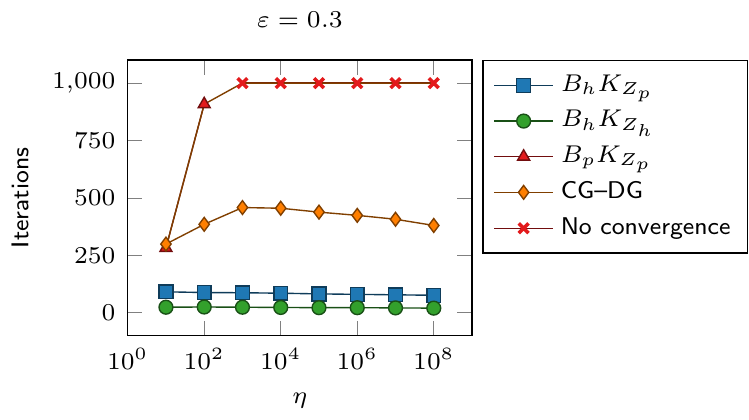}

   \DeclareRobustCommand\tikzmarker{\tikz[baseline=-0.5ex]{\node[red,line width=1pt,mark size=3pt] {\pgfuseplotmark{x}};}}

   \caption{
      Dependence of iteration counts on DG penalty parameter.
      No convergence in fewer than 2{,}000 iterations is indicated by a red ``\tikzmarker'' symbol.
   }
   \label{fig:dg-penalty}
\end{figure}

\section{Conclusions}
\label{sec:conclusions}

In this work, we have presented a framework for the construction of spectrally equivalent low-order-refined discretizations using interpolation and histopolation operators with Gauss--Lobatto points.
Simple one-dimensional norm equivalence properties of these operators can be combined using tensor-product arguments to give natural norm and seminorm equivalences in all spaces of the de Rham complex.
As an immediate consequence, we obtain spectral equivalence for the mass and stiffness matrices in $H^1$, $\Hcurl$, and $\Hdiv$, using Lagrange, \Nedelec, and Raviart--Thomas elements, independent of polynomial degree $p$ and mesh size $p$.
We additionally present a novel piecewise constant discontinuous Galerkin discretization that is spectrally equivalent to the high-order interior penalty DG discretization, independent of $p$, $h$, and penalty parameter.
This low-order discretization is equivalent to a certain weighted graph Laplacian, for which we demonstrate efficient algebraic multigrid convergence.
We use the efficient and highly scalable algebraic multigrid methods from \textit{hypre}, built using the low-order discretizations, to obtain matrix-free solvers for the high-order finite element problems; for $\Hcurl$ and $\Hdiv$ problems, we use the AMS and ADS algebraic solvers.
The effectiveness of these preconditioners on a number of three-dimensional problems is studied.
These problems possess challenging features such as coefficients with large contrasts and highly distorted geometries.
The theoretical properties of the spectrally equivalent low-order discretizations are verified.
Additionally, we demonstrate significant speedups and memory savings using the proposed solvers, in particular at higher orders.
For discontinuous Galerkin discretization, the new method proposed compares favorably to techniques that make use of the conforming subspace as a preconditioner.
Although not a focus of the present paper, these methods are highly amenable to GPU acceleration, which we anticipate to be the topic of future work.

\section{Acknowledgments}

The authors thank V.~Dobrev for insightful comments and suggestions.
This work was performed under the auspices of the U.S. Department of Energy by Lawrence Livermore National Laboratory under Contract DE-AC52-07NA27344 and was supported by the LLNL-LDRD Program under Project No.\ 20-ERD-002 (LLNL-JRNL-831792).
Sandia National Laboratories is a multimission laboratory managed and operated by National Technology and Engineering Solutions of Sandia, LLC, a wholly owned subsidiary of Honeywell International Inc., for the U.S. Department of Energy’s National Nuclear Security Administration under contract DE-NA0003525.

\printbibliography

\end{document}